\documentclass[11pt, a4paper]{article}

\usepackage[intlimits]{amsmath}
\usepackage{amsfonts,amssymb}
\usepackage[utf8]{inputenc}
\usepackage[english]{babel}
\usepackage{amsthm}
\usepackage{verbatim}
\usepackage{graphicx}
\usepackage{transparent}
\usepackage{url}

\usepackage{color}
\usepackage{xcolor}
\usepackage{hyperref}

\newcommand{\Rb}{\mathbb R}
\newcommand{\Nb}{\mathbb N}
\newcommand{\Zb}{\mathbb Z}

\newcommand{\eps}{\varepsilon}

\newcommand{\lb}{\left(}
\newcommand{\rb}{\right)}

\newcommand{\diff}[2]{\frac{\partial #1}{\partial #2}}

\newcommand{\enbrace}[1]{\lb #1 \rb}

\newcommand{\setdef}[2]{\left\{ #1\ \left|\ #2 \right.\right\}}

\newcommand{\seq}[1]{ \{ #1 \} }

\newcommand{\vmod}[1]{\left| #1 \right|}

\newcommand{\Zplus}{\Zb_{\geq 0}}

\newcommand{\Zplusinf}{\bar{\Zb}_{\geq 0}}

\numberwithin{equation}{section}

\newcommand{\dvc}{ {v'(c)} }
\newcommand{\dvcj}{ {v'(f^j(c))} }

\newcommand{\maxv}{\Gamma_0} 
\newcommand{\maxvprime}{\Gamma_1 } 
\newcommand{\maxdv}{\Gamma_1 } 
\newcommand{\mindvcj}{\Gamma_{1,c}} 
\newcommand{\maxdf}{ \Lambda_1 }
\newcommand{\maxddf}{ \Lambda_2 }

\newcommand{\maxdfper}{ \Lambda_{1,c} }

\newcommand{\powk}{ k_0 }

\newcommand{\Cbdistort}{C_1 }
\newcommand{\Csecest}{ C_2 }
\newcommand{\Cdv}{\Gamma_2}
\newcommand{\Cdvtmp}{\Gamma }

\newcommand{\Ctaylorthpower}{C_3 }

\newcommand{\holder}{H\"{o}lder }

\newcommand{\mfd}{S^1}

\newcommand{\vmodif}{\tilde{v}}
\newcommand{\pinconst}{\varkappa}
\newcommand{\prexp}{r}

\newcommand{\code}{\tilde{\pi}} 
\newcommand{\codeE}{\pi}   

\newcommand{\nbr}{n_0}   
\newcommand{\invbr}{f^{-1}}   

\newcommand{\dfx}[1]{ (f^{#1})' (x) }
\newcommand{\dfxx}[2]{ (f^{#1})' (#2) }
\newcommand{\dfxp}[1]{ f'(#1) }            

\newcommand{\ddif}[1]{ #1''}
\newcommand{\dfxpdb}[1]{ \ddif{ f } \enbrace{#1} }
\newcommand{\dfxdb}[2]{  \ddif{ (f^{#1} ) } \enbrace{#2} }

\newcommand{\dfxth}[1]{ ( \dfx{#1}   )^\theta }
\newcommand{\dfxoth}[2]{ ( \dfxx{#1}{#2})^{1-\theta} }  
\newcommand{\dfxthx}[2]{ ( \dfxx{#1}{#2} )^\theta }

\newcommand{\dfxthxp}[1]{ ( \dfxp{#1} )^\theta }

\newcommand{\lfth}{  1 -\lambda^{ -\theta  }  }
\newcommand{\lfthk}{  1 -\lambda^{ -\powk\theta  }  }
\newcommand{\lfoth}{ 1 -\lambda^{ \theta-1 } }
\newcommand{\lfothk}{ 1 -\lambda^{ -\powk(1-\theta) } }

\newcommand{\lfboth}{ 1 -b^{ \theta-1 } }

\DeclareMathOperator{\img}{Image}
\DeclareMathOperator{\dist}{dist}
\DeclareMathOperator{\diam}{diam}

\DeclareMathOperator{\argmin}{argmin}
\DeclareMathOperator{\argmax}{argmax}

\theoremstyle{plain}
 \newtheorem{theorem}{Theorem}
 
 \newtheorem{statement}{Statement}[section]
 \newtheorem{lm}[statement]{Lemma}
 \newtheorem*{lmnon}{Lemma}
 
 \newtheorem{prop}[statement]{Proposition}
 \newtheorem{corollary}[statement]{Corollary}
\theoremstyle{definition}
    \newtheorem{deff}{Definition}
    
\theoremstyle{remark} \newtheorem{rem}[statement]{Remark}

\begin{titlepage}
    \date{}

    \title{ \holder properties of 
         Weierstrass-like  solutions
        of $\theta$-twisted cohomological equations }

    \author{Dmitry Todorov \\
    {dmitry.todorov@ens.fr} \\
 \multicolumn{1}{p{.7\textwidth}}{\centering\emph{
     D.M.A., UMR 8553, \'{E}cole Normale Sup\'{e}rieure, Paris, France;  \\
 Department of Mathematics and Mechanics, Saint Petersburg State University, Saint Petersburg, Russia
 }}
    } 
\end{titlepage}

\begin{document}


\maketitle

\begin{abstract}
    It is proved that bounded solutions
    of modified ($\theta$-twisted) cohomological equations
    for expanding circle maps
    are $\theta$-\holder continuous 
    but are not $(\theta+\gamma)$-\holder continuous
    for every $\gamma>0$ at almost every point.
    This gives new examples of 
    ``nonlinear'' Weierstrass-like functions
    for which the optimal \holder exponent at most points is known. 
\end{abstract}




\tableofcontents

\section{Introduction}

In 1895  Weierstrass \cite{WEIER} constructed an example of a
continuous nowhere differentiable function
\begin{equation*}
    W(x) = \sum_{n=0}^{\infty} a^n \cos (2 \pi b^n  x)
    , \quad x\in \Rb,
    \label{eq:class_weier}
\end{equation*}
where $0<1/b<a<1$.
Hardy \cite{HARDY} 
for
$\theta = - \log a / \log b$
proved that
at any point $x$
this function
is $\theta$-\holder 
but
it is not $(\theta+\gamma)$-\holder for any $\gamma > 0$.

Hardy's result has been generalized \cite{BOUSCH_MANUSCRIPT} to 
the case when $\cos$ is replaced by any 
non-constant Lipschitz function $v$ from an open dense
subset of the space of
1-periodic Lipschitz functions $g:\Rb \to \Rb$ (with its standard norm).



Some other properties and
examples of Weierstrass-like functions
were 
considered in
\cite{ANTI_HOL_WAVELETS, WRANDOM_PHASE, ANFRAC,
 SPACE_FILL,WZYG,BOUSCH_MANUSCRIPT, CALOR_BOUSCH}.

There was a parallel research in the fractal dimension
of graphs of functions like $W(x)$
(see \cite{KAPLAN_YORKE,HAUS_GRAPHS,HAUS_SETS,BARAN_DIM_W, KELLER_DIM,SHEN}).
%
%
It is known that 
the bounds 
on a fractal dimension (either box-counting or Hausdorff)
of the graph of a function
give bounds 
on the global \holder exponent
of the function
 (see \cite{FALCONER_FRAC,HAUS_GRAPHS,ANTI_HOL_WAVELETS}, for example). 

It is important that all the research mentioned above  
applies only to the linear case when the multipliers 
in front of $\cos$ (or its replacement) and 
the one in the argument of $\cos$ are 
exactly $n$-th powers of some numbers.

For $0<\theta\leq 1$, 
$C^2$-smooth 
$f:S^1 \to S^1$ with $f'(x) > \lambda > 1$,
$\forall x\in S^1$
and  $v\in C^{1+\eps}(S^1)$ 
consider
\begin{equation} \label{eq:alpha_formula_intro}
    \alpha(x) = -\sum_{i=0}^{\infty}
    \frac{v(f^i(x)) }{ \dfxth{i+1} }
    .
\end{equation}
Functions of this kind are similar to the Weierstrass one (if one 
considers $S^1$ as $\Rb/\Zb$)
and for $\theta=1$ they correspond to solutions
of twisted cohomological equations \eqref{eq:thetaTwistCohom} 
which 
are important in the study of
linear response for one-dimensional chaotic 
dynamical systems (see \cite{BALADI_SMANIA_PE_LINRESP}).

For $\theta=1$ the modulus of continuity of $\alpha$ is
thoroughly studied in \cite{AMANDA_THESIS}.
When $f$ is an Anosov diffeomorphism,
\holder properties of $\alpha$ 
are studied in \cite{WALK_ANOSOV}. 

Note that Weierstrass-like functions
are related to certain two-dimensional
discrete dynamical systems.
In particular to compute fractal 
dimensions of the graph of $W$
it is useful to interpret it 
as a repellor of following dynamical system (see \cite{BARAN_DIM_W},
for example):
\begin{gather*}
G:(\Rb/\Zb) \times \Rb \to (\Rb/\Zb) \times \Rb, \\
    G(x,y) = \enbrace{b x (\mod 1), \frac{y - v(x)}{a} }.
\end{gather*}
Papers \cite{BEDFORD_BOX,BEDFORD,OTANI} use
this approach and give expressions
for fractal dimension(s)
of the graph of $\alpha$ for the case of one-dimensional
(piecewise)-expanding maps. 
However these expressions involve 
quantities from thermodynamical
formalism 
that are difficult to compute explicitly in general.
 In fact even if dimension can be computed explicitly, 
 if one is interested in knowing 
 \holder exponents at most points (rather than
 knowing only a global exponent), information about
 fractal dimension of the graph of the function is not enough.

In this paper the study of \holder continuity properties 
of $\alpha$ for the case $0<\theta<1$ is presented.
This gives new examples of 
dynamically-defined functions that
have at almost every point 
a globally prescribed \holder exponent that
can not be improved.

The paper demonstrates two ways of studying \holder properties 
of functions like $\alpha$.
First uses the repellor representation of 
the graph of $\alpha$ and
coding of the dynamics of $f$
by a shift on a symbolic space.
The second one is rather elementary, it is based on 
bounded distortion and density of most trajectories. 

Interestingly, an upper bound 
for the optimal \holder exponent
is much simpler to obtain
by following the first way and a lower bound -- by following
the second way.


\section{General definitions and main results}

Let $f$ be a $C^{2}$-smooth endomorphism of 
$\mfd$ 
such
that $\lambda = \min_{y\in\mfd} f'(y)  > 1$.
Assume $0< \eps \leq 1$ and $0<\theta<1$.
Let $v:\mfd \to \Rb$ be a $C^{1+\eps}$ function. 

For a positive natural number $r$ define
$ E_r(x) = rx (\mod 1). $

The following Lemma will be
used to state the main result and
in proofs:
\begin{lm}
    \label{lm:exist_unique}
    For every natural number $\prexp \geq 1$
    there exists  only one bounded solution 
    $\alpha:\mfd \to \Rb$
    to the $\theta$-twisted cohomological equation 
    \begin{equation}
        \label{eq:thetaTwistCohom}
        v( E_r(x) ) = \alpha(f(x)) - (f'(x))^{\theta} \alpha(x).
    \end{equation}
    This solution is given by the following formula:
    \begin{equation}      \label{eq:cohomSolFormula0} 
        \alpha(x) = -\sum_{i=0}^{\infty}
        \frac{v( E_r( f^i(x) ) ) }
        { \dfxth{i+1} },\quad x\in\mfd   
        .
    \end{equation}               
\end{lm}
\begin{rem}
    Note that in the case $r=1$
    the formula above gives exactly
    the function \eqref{eq:alpha_formula_intro}
    mentioned in Introduction.
\end{rem}

\begin{deff}
    For $\omega:I\to \Rb$ its
    (optimal) \holder exponent at point $x_0$ is 
    the following quantity
    \begin{equation*}
        h_x(\omega) = \varliminf_{\eps\to 0} 
        \setdef{ \frac{\log\vmod{\omega(x)-\omega(y)}}
        {\log\vmod{x-y}}}
        {y\in B_\eps(x)}.
    \end{equation*}
\end{deff}

\begin{deff}
    For $\omega:I\to \Rb$ its
    (optimal) global \holder exponent is 
    just $h(\omega) = \inf_{x\in I} h_x(\omega) $. 
\end{deff}

\begin{theorem}
    \label{thm:tdf_main}
    Let $\alpha$ be the only bounded solution to
    equation \eqref{eq:thetaTwistCohom} for $r=1$. 
    There are only two possibilities:
    \begin{enumerate}
        \item
            $\alpha$ is a $C^1$-smooth function
        \item  \label{enum:locholexp_val}
            $h_x(\alpha) = \theta$ for Lebesgue almost-every
            $x$.
    \end{enumerate}
\end{theorem}

\begin{rem}
    In fact in possibility \ref{enum:locholexp_val}
    the inequality $h_x(\alpha) \geq \theta$
    holds everywhere. See Theorem \ref{thm:uniquesol}.
\end{rem}

\begin{rem}  \label{rem:Hardy_tdf}
    Note that the formula \eqref{eq:cohomSolFormula0} 
 for the solution to equation  
    \eqref{eq:thetaTwistCohom} 
    resembles the formula for the classical
    Weierstrass function \eqref{eq:class_weier}.

    Put $\theta = - \log a / \log b$.
    The classical result of 
    of Hardy states that
    the Weierstrass function
    is $\theta$-\holder at every $x\in \Rb$ 
    but for $\gamma>0$ is not $(\theta+\gamma)$-\holder 
    at any $x\in\Rb$.
    If one puts
    $    f(x) = E_b(x)$ and
        $v(x) = \cos (2\pi x)$, 
    then Theorem \ref{thm:tdf_main}
    implies an ``almost-every'' version
    of Hardy's result.
\end{rem}

We will split the proof of the Theorem \ref{thm:tdf_main}
into two parts.

First in the Section \ref{ssec:exp_ubound}
we prove a version of Theorem \ref{thm:tdf_main}
where we have $\leq$ sign instead of $=$ in the
possibility \ref{enum:locholexp_val}.

Next in the Section \ref{sec:direct_approach}
we prove that $h_x(\alpha) \geq \theta$ 
always and for every $x$, thus finishing the proof of 
\ref{thm:tdf_main}. 
In that section we also present a different proof of a weaker version 
of Theorem \ref{thm:tdf_main} without referring to 
symbolic dynamics.

\section{Study of \holder exponents using symbolic dynamics}

In this section we will prove most of the 
Theorem \ref{thm:tdf_main}
following ideas from \cite{BEDFORD_HOL}.
First we introduce some notations and discuss a related result
from \cite{BEDFORD_BOX}.

\subsection{Box-counting dimension of the graph of $\alpha$}

Denote $I=[0,1]$.
We identify $S^1$ and $I/\sim$ where $\sim$ identifies $0$
and $1$.
Slightly abusing notation we identify $f$ and 
its lift with respect to this factorization by $\sim$.

Let $Y = I \times \Rb$ and 
a $C^1$ 
map $G: Y \to Y $
of the form $G(x,y) = (f(x), \bar{f}(x,y) )$.
Suppose $\bar{f}(x,\cdot)$ is an 
expanding diffeomorphism of $\Rb$ and
\begin{equation} \label{eq:expansion_condition}
    \vmod{ \diff{\bar{f}(x,y)  }{y} } < \vmod{f'(x)} .
\end{equation}

Denote by $\nbr$ the topological degree of $f$.
Denote by $\seq{l_i}_{i=0}^{\nbr-1}$ the sequence 
of points dividing 
into
segments of bijectivity of $f$ such that $l_0 = 0$ and 
$l_{\nbr-1} = 1$.
As $f$ is $\nbr$-to-$1$, in particular we have that
$\img f|_{[l_i,l_{i+1}]} = I$.
Put 
    $\phi_i =  (  G|_{[l_i,l_{i+1}]} )^{-1}$
    for $0\leq i < \nbr $ and
assume
that $\phi_i$ is $C^{1+\eps}$ for every $i$.

Obviously
$\phi_i:Y\to Y$ 
and injective.
There exist functions $f_i^{-1}:I \to I_i$ (inverse branches)
and $\bar{\psi}:Y\to \Rb$ such that
$\phi_i$ can be written in the following form:
 $   \phi_i = (\invbr_i(x), \bar{\psi}_i(x,y) ) $,
where
\begin{gather*}
    \vmod{ (\invbr_i)' }\in (0,1), \quad
    \vmod{ \diff{ \bar{\psi}_i(x,y)} {y} } \in (0,1)
    .
\end{gather*}

Let $a_i,b_i,c_i:Y\to \Rb$ be such that
for $z\in Y$ we have
\begin{gather*}
    D\phi_i(z) = \begin{pmatrix}
      a_i(z) & 0 \\
      b_i(z) & c_i(z)
    \end{pmatrix}
    .
\end{gather*}
Assumption
\eqref{eq:expansion_condition}
implies that
$c_i(z) > a_i(z)$ for every $z$.


Consider the global repeller for $G$: 
\begin{equation*}
    E = \setdef{(x,y)}
    { \seq{G^n(x,y)}_{n=0}^{\infty}\text{ is bounded}} 
    .
\end{equation*}

Let $\Sigma= \seq{0,1,\ldots,\nbr-1}^{\Zplus}$ be a full one-sided shift on $\nbr$ symbols.
It is well known that one can code every point of the $I$ 
by a sequence from $\Sigma$.
Define $\code:\Sigma\to I$ and $\codeE:\Sigma\to E$ by 
\begin{gather*}
    \codeE(x) = \bigcap_{n\geq 0} \phi_{x_0} \circ \ldots \circ \phi_{x_n} (E), \\
    \code(x) = \bigcap_{n \geq 0} \invbr_{x_0} \circ \ldots \circ \invbr_{x_n} (I)
    .
\end{gather*}
As both $\phi_i$ and $\invbr_i$ are strict contractions, intersections in the definitions
of $\codeE$ and $\code$ consist of single points.

Let $\sigma$ be a left shift on $\Sigma$.
It is folklore that there exist an ergodic invariant measure 
$\mu$ on $I$ that is absolutely continuous with 
respect to the Lebesgue measure on $I$.
Denote its push-back under the action of $\code$
by $\mu_\Sigma$. 
It is known that $\code$ and $\codeE$ are 
$\mu_\Sigma$-a.e. one-to-one and conjugate the 
dynamics of $f$ on $I$ and $G$ on $E$ with the dynamics of $\sigma$ on $\Sigma$ and
$\mu_\Sigma$ is a $\sigma$-invariant measure on $\Sigma$.

The repellor $E$ admits the following characterization:
\begin{lm}[\cite{BEDFORD_BOX}]
    \label{lm:E_is_graph}
    $E$ is a graph of continuous function $\alpha:I\to \Rb$.
\end{lm}

For $n\geq 0$ and $x'_0,\ldots, x'_n\in \seq{0,\ldots,\nbr-1}$
denote the corresponding $n+1$ cylinder by
\begin{equation*}
    C_{x'_0\ldots x'_n} = \setdef{x\in\Sigma}{x_i = x'_i,\ 0\leq i\leq n}.
\end{equation*}

For $\beta:\Sigma \to \Rb$ denote
\begin{equation*}
    S_n \beta (x)  = \sum_{i=0}^{n-1} \beta(\sigma^i(x) ).
\end{equation*}
For a continuous $\beta:\Sigma\to \Rb$
denote the  topological pressure of $\beta$
 by $P(\beta)$.  I.e.
\begin{equation*}
    P(\beta) = \lim_{n\to \infty} \frac{1}{n} \log
    \enbrace{ \sum_{C_n } \inf_{x\in C_n} \exp(S_n\beta(x))  },
\end{equation*}
where the summation is taken over all $n$-cylinders $C_n$ of $\Sigma$.

Define two functions $f_W,f_H:\Sigma \to \Rb$ by
\begin{gather*}
    f_W(x) = -\log a_{x_0} (\codeE \sigma(x) ), \quad
    f_H(x) = -\log c_{x_0} (\codeE \sigma(x) ).
\end{gather*}

Denote the fixed points of $\phi_0$ and $\phi_{\nbr}$
by $z_0 = (0,y_0)$ and $z_{\nbr-1} = (1,y_{\nbr-1}) $. 
Denote the (global) strong stable manifold
of $\phi_i$ at point $z_i$
by
\begin{equation*}
    W^{ss}_{\phi_i}(z_i) = Y \cap \bigcup_{n\geq 0} 
    \phi_i^{-n} (W^{ss}_{\phi_i,loc}(z_i) )  ,
\end{equation*}
where 
\begin{equation*}
    W^{ss}_{\phi_i,loc}(z_i) = \setdef{z\in Y}
    { 
        \seq{\log\vmod{\phi_i^n(z)-z_i} - n\log a_i(z_i)}_{n=0}^\infty
        \textrm{ is bounded}
    }
    .
\end{equation*}
Is not difficult to see that 
$W^{ss}_{\phi_i}(z_i)$ is a graph of 
some function from $I$ to $\Rb$ having 
a continuous 
derivative (see \cite{BEDFORD_BOX}, p. 58 for
this remark).


For completeness we give a definition of the box-counting 
dimension although we will not use its precise form:
\begin{deff}
    For $\gamma>0$ 
    and $A\subset Y$
    let $M(\gamma,A)$
    be the minimum number of boxes of 
    side length $\gamma$ that are required to cover
    $A$.
    The box dimension (or capacity) of the set
    $A$ is the following quantity
    \begin{equation*}
        \dim_B(A) = \varlimsup_{\gamma\to 0} 
        \frac{\log M(\gamma,A) }{-\log \gamma}
        .
    \end{equation*}
\end{deff}

The following theorem is valid:
\begin{theorem}[\cite{BEDFORD_BOX}] \label{thm:Bedford_dim}
    There are only two possibilities:
    \begin{enumerate}
        \item          \label{enum:wss_coincide}
            stable manifolds of all $z_i$
            coincide, i.e.
            \begin{equation*}
                W^{ss}_{\phi_0}(z_{0}) =  \cdots =
                W^{ss}_{\phi_{\nbr-1} }(z_{\nbr-1})  .
            \end{equation*}
            Then $E = W^{ss}_{\phi_0}(z_0)$. In particular $E$ is a $C^1$-smooth manifold.
        \item
            otherwise the box-counting dimension of $E$ is equal to 
            $t+1$, where $t\in \Rb$ is
            such that $P(-tf_W-f_H)=0$.
    \end{enumerate}
\end{theorem}



Recall that $v$ is a function from $I$ to $\Rb$
that is $C^{1+\eps}$ and choose a very special
$G$:
\begin{gather}
    \label{eq:Gform}
    G(x,y) = \enbrace{ f(x), y (f'(x))^{\theta} +v(x)   }
    .
\end{gather}

Then the following lemma shows how the repeller
of $G$ is related ot the Weierstrass-like function $\alpha$
mentioned above:
\begin{lm}
    %
    The repellor $E$ coincides with the graph of 
    the only bounded solutions to the
    $\theta$-twisted cohomological equation
    (Equation \eqref{eq:grapheq}).
\end{lm}
\begin{proof}
    It is enough to solve for $\alpha(x)$ the following equation
    \begin{equation} \label{eq:grapheq}
        G(x,\alpha(x) ) = (f(x), \alpha(f(x)) ).
    \end{equation}
    Its solution exists
    if and only if $f$
    satisfies  
    equation \eqref{eq:thetaTwistCohom} for $r=1$.
    Therefore it has unique bounded continuous solution
    by Lemma \ref{lm:exist_unique}.

    As  $\alpha$ satisfies Equation \eqref{eq:grapheq}, 
    its graph is invariant.
    Its boundedness implies that it is a subset of $E$. 
    But $E$ is itself a graph of a 
    function due to Lemma \ref{lm:E_is_graph}, 
    in particular for any $x\in I$
    there is only one point from $E$ with $x$ 
    as a first coordinate. Therefore
    the graph of $\alpha$ is equal to $E$.
\end{proof}


Our choice of $G$ implies that for $(x,y)\in Y$
\begin{equation*}
    \phi_i(x,y) = (f_i^{-1}(x),  y (f'(f_i^{-1}(x)) )^{-\theta}  + v(f_i^{-1}(x))  ).
\end{equation*}
Therefore
\begin{gather*}
    a_i(x,y) = (f'(f_i^{-1}(x) ) )^{-1} ,\\
    b_i(x,y) =  \diff{ 
        \enbrace{
        (f'(f_i^{-1}(x)) )^{-\theta}
        \enbrace{y   - v(f_i^{-1}(x)) } } }{x}, \\
    c_i(x,y) =  (f'(f_i^{-1}(x) ) )^{-\theta}  
\end{gather*}
and
\begin{gather*}
    f_W(x) = \log f'
    ( f_{x_0}^{-1}(\code ( \sigma(x)) ) ), \quad
    f_H(x) = \theta\log f' 
    ( f_{x_0}^{-1}(\code ( \sigma(x)) ) ). 
\end{gather*}
It is very important that in our case 
(when $G$ is of form \eqref{eq:Gform}) $f_H$ is proportional 
to $f_W$.

Recall that the map $\code$ conjugates the dynamics of the shift and the dynamics of $f$ and is bijective almost everywhere.
Also our definition of $\code$ implies that 
$\code(x)\in \invbr_{x_0}(I)$.
    Thus for $\mu_\Sigma$-a.e. $x$  we have
  $     f_{x_0}^{-1}(\code ( \sigma(x)) ) = 
       f_{x_0}^{-1}( f(\code ( x) ) = 
       \code(x) $.
Consequently
\begin{gather*}
    f_W(x) = \log f' ( \code(x) ), \quad
    f_H(x) = \theta\log f' ( \code(x) ). 
\end{gather*}

\begin{rem} \label{rem:boxdim}
    It is well known that 
    the topological pressure of 
    the potential $-\log f'$ is equal to $0$.
Properties of $\code$ imply that 
    we can view $-tf_W-f_H$ as $-(t+\theta)\log f'$.
    Therefore if $t+\theta = 1$ then 
    $t = 1-\theta$ and the box-counting dimension is equal
    to $2-\theta$ by Theorem \ref{thm:Bedford_dim}.
\end{rem}

The following standard lemma establish relation between box-counting dimension
of the graph of the function and the global \holder exponent of it.
Its proof can be found in \cite{BEDFORD_HOL}, for example.
\begin{lm} \label{lm:dim_and_exp}
    Suppose $\omega:I\to \Rb$ has the global 
    \holder exponent $\gamma$. 
    Then
     $   \gamma \leq 2- d $,
    where $d$ is the box-counting dimension of the graph of $\omega$. 
\end{lm}

Summarizing, Theorem \ref{thm:Bedford_dim} 
via Remark \ref{rem:boxdim} and Lemma \ref{lm:dim_and_exp}
implies the following:
\begin{corollary}   \label{cor:boxdim_conseq}
    If the strong stable manifolds of $z_i$
    do not coincide,
    the global
    \holder exponent of function
    $\alpha$ is less or equal than $\theta$.
\end{corollary}

\begin{rem}
    Note that this is a global result, i.e. it
    leaves a possibility that there is only a few points where 
    the \holder exponent of $\alpha$ can not be improved.
\end{rem}

\subsection{Proof of the upper bound for the \holder exponent}
\label{ssec:exp_ubound} 

Now we have to finish 
preparations for the proof of Theorem \ref{thm:tdf_main}
that implies an ''almost every'' version of
Corollary \ref{cor:boxdim_conseq}.

Denote for $x\in \Sigma$
\begin{gather*}
    E_n(x) = \phi_{x_0} \circ \ldots \circ \phi_{x_n} (E) 
    ,\\
    I_{n}(x) = \invbr_{x_0} \circ 
    \ldots \circ \invbr_{x_n} (I) .
\end{gather*}

For a set $A\subset Y$
denote
\begin{gather*}
    \vmod{A}_W = \sup\setdef{ \vmod{x-x'} }
    { (x,y),(x',y') \in A},
    \\
    \vmod{A}_H = \sup\setdef{ \vmod{y-y'} }
    { (x,y),(x',y') \in A}.
\end{gather*}

\begin{lm}[Proposition 8 from \cite{BEDFORD_BOX}] 
    \label{lm:height}
    If the possibility \ref{enum:wss_coincide}
    from Theorem \ref{thm:Bedford_dim}
    is not realized then
    there exists $N>0$ such that for 
    almost every $x\in \Sigma$ and $n\geq 0$
    \begin{equation*}
        \vmod{E_n(x)}_H \in [N^{-1},N] \exp(-S_nf_H(x))
        .
    \end{equation*}
\end{lm}

\begin{lm} \label{lm:Birkhoff_conseq}
    If the possibility \ref{enum:wss_coincide}
    from Theorem \ref{thm:Bedford_dim}
    is not realized then
    for $\mu_\Sigma$-a.e. $x\in\Sigma$    
    \begin{equation*}
        \frac{\log \vmod{E_m(x)}_H}
        {\log \vmod{E_m(x)}_W} \to \theta
        ,\quad m\to \infty.
    \end{equation*}
\end{lm}
\begin{proof}
    Note that
    \begin{equation*}
        \vmod{E_m(x)}_W = 
        \diam I_{m }(x) 
        = a_{x_0} (x) \cdots a_{x_{m}}( \sigma^{m}(x) ).
    \end{equation*}
    Then Birkhoff ergodic theorem implies 
    that
    \begin{equation*}
        \frac{\log \vmod{E_m(x)}_W   }{m+1}
        = S_m f_W(x) \to 
        \int f_W(x) d\mu_\Sigma(x) .
    \end{equation*}
    Analogously Lemma \ref{lm:height} implies that
    \begin{equation*}
        \frac{\log \vmod{E_m(x)}_H   }{m+1}
        \to 
        \int f_H(x) d\mu_\Sigma(x)
        .
    \end{equation*}
    The statement now follows immediately due 
    to the definition of functions $f_W,f_H,a_i,c_i$.
\end{proof}


The following proposition together
with Theorem \ref{thm:Bedford_dim}
implies 
a version of Theorem \ref{thm:tdf_main}
where we have $\leq$ sign instead of $=$ in the
possibility \ref{enum:locholexp_val}.

Note that a similar proposition was proved
in \cite{BEDFORD_HOL} (it is a part Theorem 9 there) 
but there $b_i$ was asked to depend only 
on the first coordinate
which is not true for our choice of $G$.

\begin{prop}
    \label{prop:exp_ubound}
    If the possibility \ref{enum:wss_coincide}
    from Theorem \ref{thm:Bedford_dim}
    is not realized then
    for Lebesgue-almost every $x\in I$
    the \holder exponent of the function $\alpha$ 
    at point $x$
    is less or equal than $\theta$.
\end{prop}

\begin{proof}
    Suppose $x\in I$ is such that $h_x(\alpha) > \theta$.
    We will show that the Lebesgue measure 
    of such $x$ is zero.

    Choose $\theta<\gamma<\beta<\eta<h_x(\alpha)$.
    Then there exists a neighborhood $U$ of $x$
    such that for every $y\in U$
    $\vmod{\alpha(x)-\alpha(y) } < \vmod{x-y}^\eta < 1$.
    Take $m$ large enough so that 
    $x\in I' = I_{m}(\code^{-1}(x) ) \subset U$.

    Put $a = \argmin_{u\in  I'}$, 
    $b = \argmax_{u\in  I'}$.
    Then taking a larger $m$ if needed 
    \begin{gather} \label{neq:height_estim}
        \vmod{E_m(\code^{-1}(x) ) }_H = 
        (\alpha(b)-\alpha(a)) \leq \vmod{b-a}^\eta \leq
        \\
        \leq
        (\diam I')^\eta = 
        \vmod{E_m(\code^{-1}(x)) }_W^\eta 
        <
        \vmod{E_m(\code^{-1}(x)) }_W^\beta
        \nonumber
        .
    \end{gather}

    Lemma \ref{lm:Birkhoff_conseq} implies that
    for $\mu_\Sigma$-a.e. $x\in \Sigma$ for $m$ large enough
    \begin{equation*}
        \frac{\log \vmod{E_m(x)}_H}
        {\log \vmod{E_m(x)}_W}  \leq \gamma
        .
    \end{equation*}

    Therefore
        $\vmod{E_m(x)}_H \geq \vmod{E_m(x)}_W^\gamma$.
    This contradicts inequality \eqref{neq:height_estim}.
\end{proof}

Note that if $b_i$ does not depend on the second coordinate
one can prove that $\theta$ is also an upper bound (see 
\cite{BEDFORD_HOL}). However this proof is more 
technical and we will prove the upper bound in Section
\ref{sec:direct_approach}
 using a simpler approach.

\section{ ``Direct'' study of \holder exponents }
\label{sec:direct_approach}

Let $f$ be a $C^{2}$-smooth endomorphism of $\mfd$
such that $\lambda = \min_{y\in\mfd} f'(y)  > 1$.
Let $v:\mfd \to \Rb$ be a 
$C^{1}$ function. 
    
Here is the missing lower bound for the 
\holder exponent.
Thus this theorem completes the proof of 
Theorem \ref{thm:tdf_main}.

\begin{theorem}   \label{thm:uniquesol} 
    For every positive natural number $r$
    the only bounded solution to 
    equation \eqref{eq:cohomSolFormula}
    is $\theta$-\holder at every point.

    In other words $h_x(\alpha) \geq \theta$
    for every $x\in S^1$.
\end{theorem}

\begin{proof}
    The theorem follows from 
    the first part of Proposition \ref{thm:thetwist_expcircle}.
\end{proof}

    Denote $\maxdf = \max_{y\in \mfd} f'(y) $.

\begin{deff}
    We say that $f$ is pinching with constant $\pinconst>0$
    or just pinching
    if
    $    \pinconst = \maxdf / \lambda^2 < 1$.
\end{deff}

The following theorem is weaker version
of Theorem \ref{thm:tdf_main}.
\begin{theorem}  \label{thm:main}
     Suppose 
     $v\in C^2(\mfd)$ and not constant.
     Then there exist $0<\pinconst_0\leq 1$ and a natural number
     $\prexp\geq 1$
     such that
     if $f$ is pinching with constant $\pinconst<\pinconst_0$
     then the solution $\alpha$ to 
     Equation \eqref{eq:thetaTwistCohom}
    given by the formula \eqref{eq:cohomSolFormula}
    is not $(\theta+\gamma)$-\holder for every $\gamma>0$
    at almost every point.

    In other words $h_x(\alpha) \leq \theta$
    for almost every $x\in S^1$.
\end{theorem}

\begin{rem}
    This theorem is a 
    consequence of 
    Proposition \ref{prop:exp_ubound}
    but we will prove it using a completely different
    method.
\end{rem}

\begin{rem}  \label{rem:Hardy0}
    In contrast with Remark \ref{rem:Hardy_tdf}
    here if one puts
        $r = 1$, $f(x) = E_b(x)$ and  
        $v(x) = \cos (2\pi x)$, 
    then Theorem \ref{thm:main}
    implies an ``almost-every'' version
    of Hardy's result \textbf{only if $b$ is large enough}.

    We discuss it in more detail in Remark \ref{rem:Hardy}.
\end{rem}

\begin{rem}
    Note 
    that the pinching requirement is necessary 
    at least for $\prexp = 1$
    because
    taking $\phi$ to be an $\omega$-\holder
    function one can use the 
    formula for equation \eqref{eq:thetaTwistCohom}
    to \emph{define} $v$  
    by
    formula \eqref{eq:thetaTwistCohom}:
    $v(x) = \alpha(\phi(x)) - (f'(x))^{\theta} \phi(x)$.
    As equation \eqref{eq:thetaTwistCohom} has only one bounded solution,
    this implies that this solution is equal to $\phi$ and is 
    therefore $\omega$-\holder. As $\omega$ can be taken arbitrarily
    it shows that the conclusion of 
    Theorem \ref{thm:main} does not hold.

    In other words the condition on a pinching constant
    is needed to ensure that possibility 
    \ref{enum:locholexp_val} from 
    Theorem \ref{thm:tdf_main} is not realized.
\end{rem}

\subsection{Intermediate results}

Theorem \ref{thm:main}
will follow from a more general technical Proposition \ref{thm:thetwist_expcircle}.

First we state a simple lemma (see \cite{MANE_ERGTH} p169 for details)
to introduce a constant that will be used
later.

\begin{lm}[Distortion estimate]  \label{lm:bdistort}
There exists a constant $\Cbdistort\geq 1$
such that for
every $x\in\mfd$,
every natural $N$ and every $0<h<1/2$ such that
$    h \leq ({\Cbdistort} {\dfxx{N}{x} })^{-1} $,
for every $\sigma \in \{ -1,1 \} $
the following estimates hold:
\begin{equation*}
    \frac{1}{\Cbdistort} \leq
    \frac{\dfxx{N}{x}}{\dfxx{N}{x+\sigma h}}
    \leq \Cbdistort 
    .
\end{equation*}
\end{lm}
    Let $\Cbdistort$ be a constant from Lemma \ref{lm:bdistort}.
    Denote 
    \begin{gather*}
        \maxv = \max_{y\in \mfd} \vmod{v(y)},\ 
        \maxvprime = \max_{y\in \mfd} \vmod{v'(y)}
        .
    \end{gather*}


\begin{deff}
    For a function $\phi:\mfd \to \Rb$
    and $\gamma>0$
    we say that $C>0$ is a local $\gamma$-\holder
    constant for $\phi$ at point $y$ if
    it is the infimum of $C'>0$ such that
    for every $0<h<1$
    the following estimate holds:
    \begin{gather}
        \vmod{\phi(y+{h})  - \phi(y)  }
            \leq C'         
            {{h}^{\gamma}} 
            .
                        \label{eq:vDiff}            
    \end{gather}
\end{deff}

\begin{rem} \label{rem:v_modif}
    Let 
    $r>0$ be a positive natural number and put
   $     \vmodif(x) = v( E_r(x) )$.
    If $C$ is a local $\eps$-\holder
    constant for $v'$ at point $y$ then
    $Cr^{1+\eps}$ is a local $\eps$-\holder
    constant for $\vmodif'(x)$.
\end{rem}


    Let $\eps>0$.
    Denote by $\Cdv$ the supremum of  
        the local $\eps$-\holder constants for $v'$ 
        over all $x\in\mfd$. It is finite 
        if $v$  is $C^{1+\eps}$.

    Let $\powk$ be a positive natural number. 

    \begin{deff}    \label{def:acond}
    Let $c$ be a point of $\mfd$
    and $v$ be $C^{1+\eps}$. 
    Say that the pair $(f,v)$ satisfy condition (A) 
    for $c\in S^1$ and $\powk \geq 1$
    if
    for every $0\leq j \leq \powk-1$
    all $v'(f^j(c))$ are either strictly positive or
    strictly negative
    simultaneously 
    and
    the following estimates hold
    \begin{gather}
        \frac
        {6^{1+1/\eps} \Cbdistort^{2-\theta} 
        \maxv \Cdv^{1/\eps}  }
        {(\lfthk) \mindvcj^{1+1/\eps}} 
        \frac{\maxdf^{\powk-1+\theta}}
        {\lambda^{(\powk+1)\theta} }
        \leq 1;  \label{neq:A_bd1}
        \\
        \frac{ \maxdv \Cbdistort^{2} }
        {  \lfothk  }  
        \frac{\maxdf^\theta}
        {\lambda^{\powk(1-\theta)+\theta} }
        \enbrace{
        \frac{ \maxdf }
        { \lambda }  }^{j(1-\theta)}
         \leq  \nonumber
         \\
         \leq
         \dvcj/4 
             , 
             \quad  0\leq j\leq \powk -1;
             \label{neq:A_bd2}
        \\
        \frac{\mindvcj}{6\Cdv} \leq \maxdf^{\powk-1}
        \label{neq:A_bd3} 
        ,
    \end{gather}

    where 
     $   \mindvcj = \min_{0 \leq j \leq \powk-1} 
        \vmod{v'(f^j(c))} $.
\end{deff}


\begin{rem}
    Condition (A) from Definition \ref{def:acond} 
    puts rather strict restrictions on
    the allowed level of nonlinearity of $f$.
    If $\powk>1$ then morally, to satisfy all of them
    $\theta$ has to be close to $1$, $\lambda$ should be 
    large and $f$ should be close to linear.

    If $\powk = 1$ then condition (A) asks
    for the pinching constant
    to be small enough (see the proof of Theorem \ref{thm:main}).

    The fact that a condition on the size of $\lambda$
    is sufficient to get absence of \holder continuity
    for some exponents in the linear case 
    (when $f = E_\lambda $ for natural $\lambda>1$)
    was mentioned in \cite{BOUSCH_MANUSCRIPT}.

\end{rem}

%

\begin{prop}
    \label{thm:thetwist_expcircle}

    Let $\alpha$ be
            the only bounded solution $\alpha$ of equation 
            \eqref{eq:thetaTwistCohom}.
    Then

    \begin{enumerate}
        \item
            $\alpha$ is $\theta$-\holder for every natural $r\geq 1$.
            There is an 
            upper bound 
            for the local $\theta$-\holder constant of $\alpha$ at every point.

        \item    \label{enum:thmpart_nonhol}
            Let $v$ be a $C^{1+\eps}$ function. 
            If $r=1$,
            $f$ is pinching and
            there exists a point $c\in\mfd$ such that
            $(f,v)$ satisfies condition (A) 
            from Definition \ref{def:acond}
            for $c\in\mfd$ and $\powk$
            then
            there exists a constant 
            $C_0=C_0(\theta,v,f,c,\powk)>0$
            such that
            for almost every $x\in\mfd$
            for every $\hat{h}$
            there exists $0<h < \hat{h}$ 
            such that
            the following lower bound holds:
            \begin{gather*}
                 \vmod{\alpha(x) - \alpha(x+h) }  
                \geq
                C_0 
                { h^\theta} 
                .
            \end{gather*}
            In particular, 
            for almost every $y\in \mfd$,
            the function
         $\alpha$ is not $(\theta+ \gamma)$-\holder 
        at $y$ for every $\gamma>0$.

    \end{enumerate}
\end{prop}

We will do the proof of this proposition for $\powk = 1$ in 
    Section \ref{sec:proofnonlin_large_lbd}
and for general $\powk$ in Appendix.

Now the proof of Theorem \ref{thm:main} is fairly easy.
\begin{proof}[Proof of Theorem \ref{thm:main}]

    The proof readily follows from 
    Proposition \ref{thm:thetwist_expcircle}
    for $\powk = 1$
    and Remark \ref{rem:v_modif}.

    Put $c'$ to be a point where the maximum of 
    $v'$ is attained.
    Put $\tilde{v} =  ( v \circ E_r  )$.
    For a point $c\in E_r^{-1}(c')$
    we get  $\tilde{v}'(c) = rv'(c') $.
         
    Remark \ref{rem:v_modif} allows
    to choose $r$ large enough so that 
    the third estimate from condition (A) is satisfied.
    Then if the pinching constant is small enough             
    the other two estimates from Condition (A) 
    are satisfied as well.

    Therefore $(f,v)$ satisfy 
    condition (A) for $c$ and $\powk = 1$
    and Proposition \ref{thm:thetwist_expcircle} applies.
\end{proof}

\begin{rem}
    A choice (not optimal) of 
    constants $C_0$ and $\delta_2$ can be written explicitly.
\end{rem}

\begin{rem}
    Note that the absence of \holder continuity at
    almost every point does not automatically
    imply absence of \holder continuity
    at every point since the values of $h$ for which
    a lower bound as above could be written
    can strongly depend on a point.
\end{rem}

\begin{rem}
    Note that condition (A) does not imply pinching 
    automatically and vice versa.
\end{rem}


It is also possible to state 
a theorem very similar to
part \ref{enum:thmpart_nonhol} 
of the Proposition \ref{thm:thetwist_expcircle}
that guarantees lower bound for 
a different 
set of points (some of which
may not belong to the set of
full measure from the Proposition \ref{thm:thetwist_expcircle}),
with different quantifiers.

\begin{prop}
    Let $v:\mfd \to \Rb$ be a 
    $C^{1+\eps}$ function. 

    If $r=1$, $f$ is pinching and
        there exists a point $c$ such that
        $(f,v)$ satisfies condition (A) at point $c$
        for power $\powk$
        then
        there exists constant 
        $C_0=C_0(\theta,v,f,c,\powk)>0$
        and $\delta_2 = \delta_2(\theta,v,f,c,\powk)>0$
        such that
        for every $\hat{h}$
        there exists a natural $N$
        such that
        for every $x\in f^{-N}(  B_{\delta_2}(c)  )$
        there exists $0<h < \hat{h}$ 
        such that
        the following lower bound holds:
        \begin{gather*}
            { \vmod{\alpha(x) - \alpha(x+h) } } 
            \geq
        C_0 { h^\theta} 
            .
        \end{gather*}

        In particular  for every $\gamma>0$
    $\alpha$  is not $(\theta+ \gamma)$-\holder at $x$.
\end{prop}

\begin{rem}
    For points $x$ that are preimages of $c$
    we can replace
    $\Cdv$ in condition (A) by 
    the maximum of local $\eps$-\holder constants for $v'$
    over $0\leq j \leq \powk -1$ at points $f^{j}(c)$.
\end{rem}

\begin{rem}
    Here is an explanation why condition (A) is needed.
    In the linear case $f = E_\lambda $     
    (this case is simpler than the general one) during the proof
    we will use decomposition 
    \begin{gather*}
       \alpha(x)-\alpha(x+h)  = \sum_{j=0}^{\powk-1} B_j(x,h)
    \end{gather*}
    and prove that 
    for every $0<\delta^{(j)}_1,\delta^{(j)}_2<\maxdf^{-j}$ 
    for some $x$
    for each $j$ there exists positive
    $h_j$ depending in an explicit way on $\delta^{(j)}_1,\delta^{(j)}_2$ such that
    \begin{gather*}
        B_j(x,h_j) h_j^{-\theta} \geq
        K(j,\delta^{(j)}_1,\delta^{(j)}_2),
        \end{gather*}
        where $K(\cdot,\cdot,\cdot)$ is
        an explicit expression (see formula \eqref{neq:final_lower_bd_nonlin}).
        To have a \emph{positive} lower bound for all $B_j(x,h)$, $0\leq j\leq \powk -1$
        for \emph{the same $h$}
        we tune $\delta^{(j)}_1,\delta^{(j)}_2$ 
        in such a way so that there exist $\delta_1,\delta_2$
        such that
        \begin{equation*}
            K(j,\delta_1,\delta_2) > 0,\quad 0\leq j\leq \powk-1
            .
        \end{equation*}
        To be able to perform this tuning, the condition (A) is required.

    \end{rem}

\begin{rem}
    It is possible to replace condition 
        that all $v'(f^j(c))$ have the same sign
        by a more complicated-looking condition meaning that 
        the sum of lower bounds for $B_j(x,h)$ for positive  $v'(f^j(c))$
        minus sum of upper bounds for $B_j(x,h)$ for non-positive ones
        is positive.
\end{rem}

\begin{theorem}  \label{cor:powk1}
     Suppose $f$ is pinching,
 $v\in C^{2+\eps}(\mfd)$ and is not constant.

    Let $c$ be the point where the maximum
    of $v'$ is achieved.
    Let $\Gamma$ is the local $\eps$-\holder constant
    of $v''$ at point $c$.
    Suppose
    \begin{gather}
            \frac{   \maxdf^\theta \Cbdistort^{2} }
            { \lambda (\lfoth)  }   
            \leq \frac{1}{4},  
            \label{neq:simple_bd1}
            \\
         {\maxdv} \leq {9\Gamma}.
            \label{neq:simple_bd2}
    \end{gather}
     then
     for almost every $x$
    the solution $\alpha$ of equation \eqref{eq:thetaTwistCohom}
    for $r=1$
     is not $(\theta+\gamma)$-\holder
     at $x$ for every $\gamma>0$.
\end{theorem}


\begin{rem}  \label{rem:Hardy}
    A relation to classical Weierestrass function
    \eqref{eq:class_weier}
    was already mentioned in Remark \ref{rem:Hardy0}.

    If one puts
        $\theta = - \log a / \log b$,
        $r = 1$,
        $f(x) = E_b(x)$, 
        $v(x) = \cos (2\pi x)$, 
    then
        $\Cbdistort = 1$, 
        $\eps = 1$, \\
         $\maxdv = v'(1/4) = 2\pi$, 
        $\Gamma = (2\pi)^3$
    and
    the bounds \eqref{neq:simple_bd1} and
    \eqref{neq:simple_bd2}
    from Theorem \ref{cor:powk1}
    transform into
    \begin{gather*}
        \frac{ 1  }
        {  \lfboth  }  
        {b^{-(1-\theta)} }
         \leq  1/4 ,
        \\
        2\pi \leq 9(2\pi)^3
        .
    \end{gather*}
    Therefore
    Proposition \ref{thm:thetwist_expcircle}
    imply an 
    ``almost every''-version of 
    Hardy's result
    for the case
     $   b \geq 
        5^{1/(1-\theta) }$.
\end{rem}


Now we prove a lemma stated before:
\begin{lmnon} [Lemma \ref{lm:exist_unique}]
    For every natural $\prexp \geq 1$
    there exists  only one bounded solution 
    $\alpha:\mfd \to \Rb$
    to equation \eqref{eq:thetaTwistCohom}.
    This solution is given by the following formula:
    \begin{equation}      \label{eq:cohomSolFormula} 
        \alpha(x) = -\sum_{i=0}^{\infty}
        \frac{v( E_r( f^i(x) ) ) }{ \dfxth{i+1} },\quad x\in\mfd  . 
    \end{equation}               
\end{lmnon}
\begin{proof}
    Fix $x\in \mfd$.
    It is easy to see that series \eqref{eq:cohomSolFormula}
    gives a bounded solution to equation \eqref{eq:thetaTwistCohom}.

    To prove that it is the only bounded solution, suppose there is another 
    bounded solution $\beta$.
    For $K = \max(\sup \vmod{\alpha}, \sup \vmod{\beta})$
    take $N\in\Nb$ so that 
    \begin{equation*}
        \frac{K}{\dfxth{N}} < \vmod{\beta(x) - \alpha(x) } / 3 .
    \end{equation*}

    Note that for $\alpha$
    we can write a finite analog of solution forumla \eqref{eq:cohomSolFormula}:
    \begin{gather*}
        \alpha(x) = \frac{-v(E_r(x)) }{ \dfxthxp{x} }  +
        \frac{\alpha(f(x)) }{ \dfxthxp{x}}
        =      \nonumber
        \\
        =\frac{-v(E_r(x)) }{ \dfxthxp{x} }  
        + \frac{1 }{ \dfxthxp{x} } 
        \enbrace{\frac{-v(E_r(f(x))) }{ \dfxthxp{f(x)} }  +
        \frac{\alpha(f^2(x)) }{ \dfxthxp{f(x)} } } = 
        \nonumber
        \\
        =\frac{-v(E_r(x)) }{ \dfxthxp{x}  }  -  
        \frac{v(E_r(f(x))) }{ \dfxth{2}  }  +
        \frac{\alpha(f^2(x)) }{\dfxth{2}}  = \ldots =  
        \nonumber
        \\
        =-\sum_{i=0}^{N-1} \frac{v(E_r(f^i(x))) }{ \dfxth{i+1} } +
        \frac{\alpha(f^N(x)) }{ \dfxth{N} }.
    \end{gather*}

    Analogously 
    \begin{gather*}
        \beta(x) = -\sum_{i=0}^{N-1} \frac{v(E_r(f^i(x))) }{ \dfxth{i+1} } +
        \frac{\beta(f^N(x)) }{ \dfxth{N} }.
    \end{gather*}
    Then 
    \begin{equation*}
        \frac{3K}{ \dfxth{N}} < \vmod{\alpha(x) - \beta(x) }= 
        \vmod{ \frac{  \alpha(f^N(x)) - \beta(f^N(x)) }
        { \dfxth{N} } } \leq \frac{2K}{  \dfxth{N} },
    \end{equation*}
    which is  a contradiction.

\end{proof}

\begin{proof}[Proof of Proposition \ref{thm:thetwist_expcircle} ]

    Let $\alpha$ be the only bounded solution to equation  
    \eqref{eq:thetaTwistCohom}.

    We show the proof first for linear $f$
    to give ideas
    and then for general (nonlinear) $f$.
    In each case we first prove an upper bound for
    $\vmod{\alpha(x)-\alpha(x+h)}h^{-\theta}$
    for every $f$, 
    then a lower bound for the case when
    $(f,v)$ satisfies condition (A) for power $\powk = 1$.
    
    The proof of the lower bound for the case of general 
    $\powk$ is put in the Appendix.
    
    Both for linear and nonlinear case
    we prove several lemmas with estimates
    and then we use them in different combinations 
    to study \holder continuity properties of $\alpha$.

    We prove upper bounds only for the case $r=1$ but 
    the case of general natural $r\geq 1$ 
    follows immediately replacing $v$ by $v\circ E_r$.

    Lemmas that will follow are meant 
    to be inside the proof of the theorem so
    they inherit notations and 
    assumptions made during the proof before
    they are stated.

    Recall that there exists a measure 
    with a positive density with
    respect to Lebesgue measure on the 
    circle that is ergodic for 
    $f$ (see \cite{MANE_ERGTH} for example). Therefore
    almost every point $x\in\mfd$ has a dense orbit.

    \subsection{Proof of Proposition \ref{thm:thetwist_expcircle} in the linear case}

Suppose first that $f(x) = E_\lambda(x) $, 
for natural $\lambda > 1$.

    For every 
    $x$ from $S^1$ and $h>0$
    consider the following decomposition:
    \begin{gather}
        \alpha(x) - \alpha(x+h)
        = 
         -\sum_{i=0}^{\infty}
         \frac{v(f^{ i}(x)) }{ \lambda^{\theta ( i+1)} } 
        +\sum_{i=0}^{\infty}
        \frac{v(f^{ i}(x+h)) }{ \lambda^{\theta (  i+1)} }  = 
        \nonumber
        \\
        = \sum_{i=0}^{\infty}
        \frac{1}{\lambda^{\theta ( i+1)} }
        \enbrace{v(f^{ i}(x+h)) -v(f^{ i}(x))}   .
    \label{eq:alphadif_lin_largelbd} 
    \end{gather}

    Denote by $\Zplus$ the set of nonnegative integers and
    $\Zplusinf = \Zplus \cup \{ \infty \}$.

    For $n\in \Zplus$ and $m \in \Zplusinf$ such that $n\leq m$
    we introduce the following notation:
    \begin{equation*}
        S_{n,m}(x,h) = 
 \sum_{i=n}^{m} \frac{1}{\lambda^{(   i+1)	\theta} }
 \enbrace{v(f^{ i }(x+h)) -v(f^{ i}(x))}     .
    \end{equation*}

    Note that for every $n\in\Zplus$ there
    exist $\xi_i\in (x,x+h)$ for $0\leq i\leq n$
    such that
    \begin{gather*}
        S_{0,n}(x,h) = \frac{h}{\lambda^{\theta}} \sum_{i=0}^n 
        v'( f^i(\xi_i) )   \lambda^{ i(1-\theta)}
    \end{gather*}

    Next we will write estimates of  \eqref{eq:alphadif_lin_largelbd}
    from above and from below with different quantifiers for $x$ and $h$.

    \subsubsection{Technical lemmas}

    \begin{lm}                         \label{lm:lin_supp_bounds}

    Let $N$ be a natural number,  
    $0<\delta_1<1$
    and $h>0$ be such that
    $  \delta_1 \leq h{ \lambda^{ N} } $.

    Then 
    for every $x\in\mfd$ the following estimate holds:
\begin{gather*}
    \vmod{S_{N+1,\infty}(x,h) }
  \leq
  \frac{ 2 \maxv }{ (\lfth) \lambda^{2\theta} \delta_1^{\theta} }
        h^\theta  
        .
  \end{gather*}

\end{lm}
\begin{proof}[Proof of the lemma]
    We may easily estimate the tail
    of ${S_{0,\infty}}$ using the lower bound on $h$:
    \begin{gather*}
        \vmod{S_{N+1,\infty}(x,h) }  \leq
        2\maxv \frac{1}{\lambda^{(N+2 )\theta}}
        \sum_{i=0}^{\infty}  \lambda^{-i \theta} 
        \leq
        \frac{2\maxv}{(1 - \lambda^{-\theta}) \lambda^{(N+2) \theta} }
        \leq  
        \frac{ 2 \maxv }{ (\lfth) \lambda^{2\theta} \delta_1^{\theta} } h^\theta .
    \end{gather*}

\end{proof}

    \begin{lm} \label{lm:lin_lower_bd}
    Let $N$ be a natural number,
    $0<\delta_1,\delta_2\leq 1$
    and $h>0$ be such that
    $    { \delta_1} 
    \leq h{ \lambda^{ N} } \leq
        { \delta_2 }  $
        .
    Then
    for every $x\in\mfd$ the following estimate holds:
    \begin{gather*}
        \vmod{ S_{0,N}(x,h) } 
        \leq
        \frac{2 \maxvprime}{\lfoth}
            \delta_2^{1-\theta}
         \lambda^{ - \theta}
        h^{\theta}
        .
    \end{gather*}

    If $x,c\in\mfd$ are 
    such that
    $ \dist(f^{  N}(x), c ) \leq \delta_2$,
    and 
    $ v'(c) > 0$
    then
    \begin{gather*}
        h^{-\theta}
        \frac{\lambda^{\theta} }{ \delta_1^{1-\theta}} 
         \vmod{ S_{0,N}(x,h) }
        \geq
            {v'(c)} 
            -
 2\Cdvtmp \delta_2^\eps   
        -
        \frac{2\maxvprime}{ (\lfoth) \lambda^{1-\theta} } 
        ,
    \end{gather*}
    where $\Cdvtmp$ is a local $\eps$-\holder constant for $v'$ at point $f^N(x)$.
\end{lm}

\begin{proof}[Proof of the lemma] 
    First estimate from the statement follows from the upper bound on $h$:
    \begin{gather*}
        h^{-\theta}
        \frac{\lambda^{\theta} }
        { \delta_2^{1-\theta} }
        \vmod{ S_{0,N}(x,h) }
        =
        h^{1-\theta} 
        \frac{\lambda^{\theta} }
        { \delta_2^{1-\theta} }
        \vmod{
             \frac{1}{\lambda^\theta} \sum_{i=0}^n 
        v'( f^i(\xi_i) )   \lambda^{i(1-\theta)} 
             } \leq 
        \\
        \leq
        \sum_{i=0}^{N} \vmod{ v'(f^{ i}(\xi_i)) }
            \frac{  1 }
            {  \lambda^{(N-i )(1-\theta)}  }     
            \leq 
            \frac{2  \maxvprime}{ \lfoth    } 
            .
    \end{gather*}

        Note that we have for certain $\xi\in (x,\xi_N) \subset (x,x+h)$
    \begin{gather*}
        \vmod{ v'(f^{  N}(\xi_N)) - 
        v'(f^{ N}(x)) }
        \leq  
        \Cdvtmp \vmod{ f^{  N}(\xi_N) - f^{  N}(x) }^\eps
        =
        \\
        =
        \Cdvtmp \vmod{ \dfxx{  N}{\xi} (\xi_N - x)}^\eps 
        =
        \Cdvtmp (\lambda^{ N} h)^\eps \leq \Cdvtmp  \delta_2^\eps
        .
    \end{gather*}
    We also have that
       $ \vmod{ v'(f^{ N}(x)  ) - v'(c) } 
        \leq 
        \Cdvtmp \delta_2^\eps$.

    Using the lower bound for $h$, estimate \eqref{eq:vDiff}
    and the condition on $x$
    we can write the following estimates:
    \begin{gather*}
        h^{-\theta}
        \frac{\lambda^{\theta} }{ \delta_1^{1-\theta}} 
        \vmod{ S_{0,N}(x,h) }
        =
        h^{1-\theta} 
        \frac{\lambda^{\theta} }{ \delta_1^{1-\theta}} 
        \vmod{
             \frac{1}{\lambda^\theta} \sum_{i=0}^n 
        v'( f^i(\xi_i) )   \lambda^{  i(1-\theta)} 
             } \geq 
        \\
        \geq
            \vmod{
                \sum_{i=0}^{N} v'(f^{ i}(\xi_i))
            \frac{  1 }
            {  \lambda^{(N- i)(1-\theta)}  } }    
            \geq 
            \\
        \geq
         \vmod{ v'(f^{ N}(\xi_N)) }  
                -
                \sum_{i=0}^{N-1} \vmod{v'(f^{ i}(\xi_i)) }
            \frac{  1 }
            {  \lambda^{ (N-i) (1-\theta)}  }     
            \geq 
            \\
        \geq  
        {v'(f^N(x))} -
        \Cdvtmp \delta_2^\eps
        -
             \frac{2\maxvprime}{(\lfoth)\lambda^{1-\theta} }
             \geq
        {v'(c)} -
        2\Cdvtmp \delta_2^\eps
        -
             \frac{2\maxvprime}{(\lfoth)\lambda^{1-\theta} }
             .
    \end{gather*}

\end{proof}

    \subsubsection{Upper bound}
    First we prove that $\alpha$ is $\theta$-\holder.

    Fix $x\in\mfd$ and
    h such that $\lambda^{-1}/2>h>0$.
    Take a natural $N$ (depending on $h$) such that
    $     1/2  \leq 
    h{ \lambda^{N} } 
    \leq 1     $.

    Then lemmas \ref{lm:lin_supp_bounds}
    and \ref{lm:lin_lower_bd} 
    for 
    $\delta_1 = 1/2$ and  $\delta_2 =1$ 
    imply that
    the following upper bound for the normalized absolute value
    of \eqref{eq:alphadif_lin_largelbd} holds:
    \begin{gather*}
        \frac{\vmod{\alpha(x+h) - \alpha(x) } }{h^\theta}
        \leq 
        \frac{ 2 \maxvprime }
        {\lambda^{-\theta}(\lfoth)}
            +
            \frac{4 \maxv }{\lambda^{2\theta} (\lfth) }
            .
    \end{gather*}
    As this bound does not depend on $h$ and $N$, it
    proves that $\alpha$ is $\theta$-\holder.
    
    

    \subsubsection{Lower bound when
    condition (A) holds for $\powk = 1$}

    Assume that condition (A) is satisfied for 
    a point $c$ and $\powk = 1$. 
    Suppose without restricting generality that $v'(c)>0$.

    Fix $\hat{h}$.

    We first give expressions for $\delta_1,\delta_2$
    that depend only on $f,v$ and $\powk$
    and later
    select $x$, $N$ and $h = h(N,\delta_1,\delta_2) < \hat{h}$
    such that $ \vmod{\alpha(x) - \alpha(x+h) } / h^\theta$ 
    has a positive lower bound.

    For every 
    $0<\delta_1,\delta_2\leq 1$
    for every $x$ for which  
    there exists a natural $N>1$ such that
    $    \vmod{f^{ N}(x) - c} < 
        \delta_2 $,
    the lemmas above imply that
        for every
        $h = h(N)$ such that 
    $    { \delta_1 } \leq h { \lambda^{ N}  } \leq
        { \delta_2 }    $,
    the following lower bound holds 
        \begin{gather}
        \frac{\vmod{\alpha(x+h) - \alpha(x) } }{h^\theta} 
             {\delta_1^{\theta-1}}
            {\lambda^{\theta}  } 
        \geq \nonumber 
        \\
        \geq
            \dvc
            -
            2\Cdv \delta_2^\eps 
        -
             \frac{2\maxvprime}{ (\lfoth) \lambda^{(1-\theta)} } 
             -
        \frac{ 2 \maxv }{ (\lfth)\lambda^{\theta}
        \delta_1 }
        \label{neq:final_lower_bd_lin_large}
        .
        \end{gather}
        Note that the right-hand side of 
        this expression does not depend on $x,N$ and $h$.

    It is possible to choose $\delta_1,\delta_2$ independently of $N,h$ and $x$ 
    in such a way that the expression above is always greater than zero.
    To guarantee that the second and the fourth summands in the right-hand
    side of \eqref{neq:final_lower_bd_lin_large} are both less than
    $\dvc/3$ it is enough to put 
    \begin{gather*}
        \delta_1 = \frac{6 \maxv}{ (\lfth) \dvc   \lambda^{\theta }},\ 
        \delta_2 = \enbrace{ \frac{\dvc}{6\Cdv}}^{{1/\eps}}
        .
    \end{gather*}

    Inequalities \eqref{neq:A_bd1},\eqref{neq:A_bd3} from condition (A)
    imply that $\delta_1\leq \delta_2\leq 1$.

    Now fix a point $x$ such that there exists
    a natural $N$ such that
    \begin{gather*}
            \vmod{f^{ N}(x) - c} < 
            \delta_2; \  
            \delta_2 \lambda^{-N} < \hat{h}
        .
    \end{gather*}
    One can take $x$ to be a point from
    $N$'th-preimage of $B_{\delta_2}(c)$ for sufficiently large $N$ or
    any point with a
    dense trajectory (the set of which has full Lebesgue measure).
    In the latter case $x$ does not depend on $\hat{h}$.

    Now inequality \eqref{neq:A_bd2} from condition (A)
    and the definition of $\delta_1,\delta_2$
    imply that we can select an $h < \hat{h}$ such that 
    \begin{equation*}
        \frac{ \alpha(x) - \alpha(x+h) } { h^\theta}
        \geq  C_0 > 0 , 
    \end{equation*}
    where
    \begin{gather*}
        C_0  = \delta_1^{1-\theta} \lambda^{-\theta} \frac{ \dvc }{12}
        =  \enbrace{ \frac{6 \maxv }{\lfth} }^{1-\theta}
        \frac{1}{12} \lambda^{-2\theta+\theta^2} 
        (\dvc)^\theta     .
    \end{gather*}
    in particular, $C_0$ does not depend on ${h}$.


    \subsection{Proof of Proposition \ref{thm:thetwist_expcircle} for general expanding $f$ }
\label{sec:proofnonlin_large_lbd}
    Now consider 
    a general $f$. 
        
      Denote
 $         \maxddf = \max_{y\in\mfd} \vmod{\dfxpdb{y}}$.

    For every 
    $x$ from $S^1$ and $h>0$,
    consider the following decomposition:
    \begin{gather}
        \alpha(x) - \alpha(x+h) = \label{eq:alphadif_nonlin_large_lbd} \\
        = -\sum_{i=0}^{\infty}
        \frac{v(f^i(x)) }{ \dfxthx{i+1}{x} } 
        +\sum_{i=0}^{\infty}
        \frac{v(f^i(x+h)) }{ \dfxthx{i+1}{x+h} }  = 
        \nonumber\\
        = \sum_{i=0}^{\infty} \frac{1}{\dfxthx{i+1}{x} }
        \enbrace{v(f^i(x+h)) -v(f^i(x))} +
        \label{eq:alphadif1smnd} \\
        + \sum_{i=0}^{\infty} v(f^i(x+h))\enbrace{\frac{1}{\dfxthx{i+1}{x+h} }
        - \frac{1}{\dfxthx{i+1}{x} } } .   \label{eq:alphadif2smnd}
    \end{gather}

    \subsubsection{Technical lemmas}
    

We will need two simple formulae

\begin{lm}
    For every $x\in \mfd$ and natural $k_1\geq k_2 \geq 0$ we have
    \begin{gather}
        \frac{ \dfx{k_1} }{ \dfx{k_2} } = 
        {  \dfx{k_1-k_2}( f^{k_2}(x) ) }   \label{eq:dfdivide} .
    \end{gather}
\end{lm}
%


\begin{lm}
    For every $x\in \mfd$ and natural $k\geq 1$ the following representation is valid:
    \begin{gather}
        \dfxdb{k}{x} =
          \sum_{p=0}^{k-1}   \dfx{k-p}   \dfxpdb{ f^{k-p}(x) }
        \frac{ \dfx{k} }{ \dfxp{ f^{k-p}(x) } } +
            \dfxpdb{x} 
            \frac{ \dfx{k} }{ \dfxp{ x } }  
            .  \label{eq:ddf_rep}
    \end{gather}
\end{lm}
\begin{proof}
    \begin{gather*}
        \dfxdb{k+1}{x} = \ddif{ \enbrace{ f(f^k(x)) } } =
        \enbrace{ \dfxp{f^k(x)} \dfx{k} }' =\\
        =
        \dfxpdb{f^k(x)} \enbrace{ \dfx{k} }^2 + \dfxp{f^k(x)} \dfxdb{k}{x} .
    \end{gather*}

    Therefore 
    \begin{gather*}
        \sum_{p=0}^{k-1}  \enbrace{ \dfx{k-p}  }^2 \dfxpdb{ f^{k-p}(x) }
        \prod_{j=1}^{p}  \dfxp{f^{k-j}(x)} + 
        \dfxpdb{x} 
        \prod_{j=1}^{k-1}  \dfxp{f^{k-j}(x)} = \\
            %
        = \sum_{p=0}^{k-1}
        \enbrace{ \dfx{k-p}  }^2 \dfxpdb{ f^{k-p}(x) }
            \frac{ \dfx{k} }{ \dfx{k-p+1} } +
        \dfxpdb{x} 
            \frac{ \dfx{k} }{ \dfxp{ x } }  .
    \end{gather*}
\end{proof}

    Now we prove an estimate for \eqref{eq:alphadif2smnd}.
    
\begin{lm}   \label{lm:nlin_fderiv_bounds}
    Let $\delta_1,\delta_2>0$.
    For every $x\in\mfd$
    for every natural number $s>0$
    there exists
    $\Csecest(s)>0$ such that 
    for every
    natural number $N\geq s$ and 
    $h>0$ such that
    \begin{gather*} 
          \delta_1 \leq h { \Cbdistort \dfxx{N}{x} }\leq
          \delta_2  
        ,
    \end{gather*}
    the following estimate holds 
    \begin{gather}
         \sum_{i=0}^{\infty}
        v(f^i(x+h))\enbrace{\frac{1}{\dfxthx{i+1}{x+h} }
        - \frac{1}{\dfxth{i+1} } }   \label{eq:alphadif2smnd_lm}
        \leq
        \\
        \leq 
        \frac{ \enbrace{\lambda^{-s\theta} + \Cbdistort^{\theta}
            \enbrace{ \frac{\maxdf}{\lambda^2}}^{s\theta} 
        }  
    \maxv }{ \lfth } 
        \delta_1^{-\theta} h^\theta 
         +
         \Csecest(s) N h 
         \nonumber
         .
    \end{gather}
\end{lm}

\begin{proof}

    Fix a natural number $s>0$.
    Split  \eqref{eq:alphadif2smnd_lm} into two parts:
    \begin{gather}
        \vmod{\sum_{i=0}^{\infty}
        v(f^i(x+h))\enbrace{\frac{1}{\dfxthx{i+1}{x+h} }
        - \frac{1}{\dfxth{i+1}}} }\leq \nonumber   \\     
        \leq
        \vmod{\sum_{i=0}^{N+s-2} v(f^i(x+h))
        \enbrace{\frac{1}{\dfxthx{i+1}{x+h} }
        - \frac{1}{\dfxth{i+1}}} } + \label{eq:1smd_nlin} \\          
        +
        \vmod{\sum_{i=N+s-1}^{\infty}
        v(f^i(x+h))\enbrace{\frac{1}{\dfxthx{i+1}{x+h} }
        - \frac{1}{\dfxth{i+1}}} } .      
        \label{eq:2smd_nlin}
    \end{gather}

    First we have to
    estimate \eqref{eq:2smd_nlin} using the lower bound on $h$,
    the distortion estimate 
    and formula \eqref{eq:dfdivide}.

    \begin{gather*}
        \vmod{\sum_{i=N+s-1}^{\infty}
        v(f^i(x+h))
        \enbrace{\frac{1}{\dfxthx{i+1}{x+h} }
        - \frac{1}{\dfxth{i+1}}} }
        \leq
        \\
        \leq 
        \maxv \vmod{\sum_{i=0}^{\infty}
        \enbrace{
            \frac{1}{\dfxthx{N+s}{x+h} \dfxthx{i}{f^{N+s}(x+h) } }
            - \frac{1}{\dfxthx{N+s}{x}\dfxthx{i}{f^{N+s}(x)  } }
        } }
        \leq
        \\
        \leq 
        \maxv \vmod{
        \frac{1 }{\dfxthx{N+s}{x}}
        \sum_{i=0}^{\infty}
        \enbrace{
 \frac{\dfxthx{N+s}{x} }{\dfxthx{N+s}{x+h}}
            \frac{1}{ \dfxthx{i}{f^{N+s}(x+h) } }
            - \frac{1}{\dfxthx{i}{f^{N+s}(x)  } }
        } }
        \leq
        \\
        \leq
        \maxv  \frac{1+\Cbdistort^\theta
            \enbrace{ \frac{\maxdf}{\lambda}}^{s\theta} 
        } 
        { \dfxth{N+s}  }
        \sum_{i=0}^{\infty} \frac{1}{ \lambda^{i\theta} } 
    \leq
    \frac{ \enbrace{1+\Cbdistort^{\theta}
            \enbrace{ \frac{\maxdf}{\lambda}}^{s\theta} 
        }  
    \maxv }{ \lambda^{s\theta} (\lfth) } 
        \delta_1^{-\theta} h^\theta 
        .
    \end{gather*}

    Now we start estimating \eqref{eq:1smd_nlin}.

    Note that we have
       $ \vmod{ (1+h)^{\theta} - 1 } \leq \Ctaylorthpower h$.

    Since $f$ is $C^2$, there exists $\xi\in (x,x+h)$
    such that
    \begin{equation*}
        \dfxx{k}{x}  -  \dfxx{k}{x+h}  =  - \dfxdb{k}{\xi} h.
    \end{equation*}

    Using formulae \eqref{eq:dfdivide}, \eqref{eq:ddf_rep}
    and the distortion estimate,
    we can write for every natural number $1 \leq k \leq N$
    \begin{gather*}
        \vmod {\frac{1}{\dfxthx{k}{x+h} }
        - \frac{1}{ \dfxth{k} } } 
        = 
        \vmod{ \frac{\dfxthx{k}{x} - \dfxthx{k}{x+h} }
        {\dfxthx{k}{x}\dfxthx{k}{x+h}} } \leq  
        \\
        \leq 
        \vmod{  \frac{  \Ctaylorthpower   \dfxdb{k}{\xi} h}
        { \dfxoth{k}{\xi}   
        \dfxthx{k}{x}\dfxthx{k}{x+h}}  } \leq 
        \\
        \leq   
        \frac{\Ctaylorthpower \maxddf h}
        { \dfxoth{k}{\xi}    \dfxthx{k}{x}\dfxthx{k}{x+h}}
        \cdot
        \\ 
        \cdot \enbrace{
            \sum_{p=0}^{k-1} 
            { \dfxx{k-p}{\xi} } 
            \frac{ \dfxx{k}{\xi} }{ \dfxp{f^{k-p}(\xi)} } +
            \frac{ \dfxx{k}{\xi} }{ \dfxp{ \xi } }  } =
            \\
            =
            \frac{\Ctaylorthpower \maxddf h} {\dfxthx{k}{x+h}} \enbrace{
                \sum_{p=0}^{k-1} 
                \frac{ \dfxx{k-p}{\xi}  }{ \dfxoth{k}{\xi} \dfxthx{k}{x}}
                \frac{ \dfxx{k}{\xi} }{ \dfxp{ f^{k-p}(\xi) } } +
                \frac{ \dfxx{k}{\xi} }{ \dfxp{ \xi } } }
                =
                \\
                =
                \frac{\Ctaylorthpower \maxddf h} {\dfxthx{k}{x+h}} 
                    \lb
                    \sum_{p=0}^{k-1} 
                    \frac{   \Cbdistort^\theta  }
                    { \dfxoth{p}{f^{k-p}(\xi) } \dfxthx{p}{f^{k-p}(x) } }
                    \frac{ \dfxx{k}{\xi}   }{\dfxp{f^{k-p}(\xi)} }
                    +     \right.
                    \\
                    \left. +
                    \frac{ \dfxx{k}{\xi} }{ \dfxp{ \xi } }  
                    \rb =
                    \frac{ \dfxthx{k}{\xi} h } {\dfxthx{k}{x+h}} 
                    \lb
                    \sum_{p=0}^{k-1} 
                    \frac{   \Cbdistort^\theta  }
                    { \dfxoth{p}{f^{k-p}(\xi) }
                    \dfxthx{p}{f^{k-p}(x) } \dfxp{f^{k-p}(\xi) } }
                    +     \right. 
                    \\
                    \left. +
                    \frac{ 1 }{ \dfxp{\xi} }  
                    \rb  
                    \leq
                    \Cbdistort^\theta \Ctaylorthpower \maxddf h \enbrace{ \sum_{p=0}^{k-1} 
                    \frac{\Cbdistort^\theta }{ \lambda^{ p(1-\theta) + p\theta + 1}  } 
                    + \frac{1}{\lambda}}
                    \leq
                    \frac{\Cbdistort^\theta \Ctaylorthpower \maxddf }{\lambda}
                    \enbrace{ \frac{\Cbdistort^\theta}{1-\lambda^{-1}} + 1} h.
                \end{gather*}
                For $k> N$ we can do a similar estimate
                using rough bounds instead of distortion estimates, i.e.
                \begin{equation*}
                    \frac{ \dfxx{N+p}{\xi} }{ \dfxx{N+p}{x} }  \leq \Cbdistort
                    \enbrace{ \frac{\maxdf}{\lambda}}^p
                    .
                \end{equation*}

                Therefore we have an
                estimate for \eqref{eq:1smd_nlin}
                :
                \begin{gather*}
                    \vmod{\sum_{i=0}^{N+s-2}
                    v(f^i(x+h))\enbrace{\frac{1}{\dfxthx{i+1}{x+h} }
                    - \frac{1}{\dfxth{i+1}}} } 
                    \leq 
                    \\
                    \leq
                \frac{\Cbdistort^\theta
                \Ctaylorthpower \maxddf}{\lambda}
                \enbrace{
                    \frac{\Cbdistort^\theta}{1-\lambda^{-1}} + 1
                } N h
              +
              \\
              +
                \frac{\Cbdistort^\theta
                \Ctaylorthpower \maxddf }{\lambda}
              \Cbdistort^\theta  \enbrace{\frac{\maxdf}{\lambda}}^{(s-2)\theta}
              \enbrace{ 
                  \enbrace{\frac{\maxdf}{\lambda}}^{(s-2)\theta}
                    \frac{\Cbdistort^\theta}{1-\lambda^{-1}} + 1
              }
              (s-2) h
              .
                \end{gather*}          

            Put $\Csecest(s)$
            to be twice the last expression divided by $h$.
            \end{proof}


    The rest is very similar to the linear case.

    For $n\in \Zplus$ and $m \in \Zplusinf$ such that 
    $n\leq m$ we introduce the following notation for \eqref{eq:alphadif1smnd}:
    \begin{equation*}
        {S}_{n,m}(x,h) = 
        \sum_{i=n}^{m} \frac{1}{\dfxthx{i+1}{x} }
        \enbrace{v(f^i(x+h)) -v(f^i(x))}     .
    \end{equation*}

    Note that for every $n\in\Zplus$ there
    exist $\xi_i \in (x,x+h)$ for $0\leq i\leq n$
    such that
    \begin{gather*}
        S_{0,N}(x,h)   
        =
        {h} 
        \sum_{i=0}^{N} v'(f^i(\xi_i))
        \frac{ \dfxx{i}{\xi_i}   }{\dfxthx{i+1}{x }} 
        =            \\
        =
        {h} 
        \sum_{i=0}^{N} v'(f^i(\xi_i))
        \frac{ \dfxx{i}{\xi_i}   }{\dfxthx{i}{x }} 
        \frac{ 1 }{ \dfxthxp{ f^i(x) } }
        .
    \end{gather*}

\begin{lm}        \label{lm:nlin_supp_bounds}
    Let $x\in\mfd$, let $N$ be a natural number,  
    $0<\delta_1,\delta_2<1$
    and $h>0$ be such that
       $ {  \delta_1} \leq h{ \Cbdistort \dfxx{N}{x} }$.
    Then 
     the following estimate holds:
\begin{gather*}
  \vmod{{S}_{N+1,\infty}(x,h) } \leq
    \frac{2 
\maxv}{ \lambda^{2\theta}
    ( \lfth ) } \delta_1^{-\theta}  
                h^{\theta}
                .
\end{gather*}
\end{lm}

\begin{proof}[Proof of the lemma]
We may easily estimate the tail
of ${S_{0,\infty} (x,h)}$ using the lower bound on $h$:
\begin{gather*}
    \vmod{{S}_{N+1,\infty} (x,h) }  \leq
    2\maxv \frac{1}{\dfxth{N+2}} \sum_{i=0}^{\infty}  \lambda^{-i\theta} 
    = \\
    =  \frac{2\maxv}{1 - \lambda^{-\theta}} 
    \frac{1}{ \dfxthx{2}{f^N(x)} \dfxthx{N}{x}}
    \leq 
    \\
    \leq
    \frac{2 
\maxv}{ \lambda^{2\theta}
    ( \lfth ) } \delta_1^{-\theta} h^\theta  
    .
\end{gather*}

\end{proof}

\begin{lm}           \label{lm:nlin_lower_bd}
    Let $x\in\mfd$, let $N$ be a natural number,  
    $0<\delta_1,\delta_2\leq 1$
    and $h>0$ be such that
     $   {  \delta_1} \leq 
        h { \Cbdistort \dfxx{N}{x} }
        \leq
        {  \delta_2 }$.
    Then
    \begin{gather*}
        \vmod{ S_{0,N}(x,h) } \leq
        \delta_1^{1-\theta}
            \frac{2 \Cbdistort^\theta \maxvprime}{ \lambda^\theta (\lfoth)    } 
        h^{\theta}
        .
    \end{gather*}

    If  $x,c\in\mfd$
    such that
    $ \dist(f^{  N}(x), c ) \leq \delta_2$,
    and 
    $ v'(c) > 0$
    then

    \begin{gather*}
        h^{-\theta}
        \frac
        { \maxdf^{\theta} \Cbdistort^{2-\theta} }
        { \delta_1^{1-\theta} }
        \vmod{ S_{N} }
        \geq
            v'(c)
            -
            2\Cdvtmp 
              \delta_2^\eps 
            -
            \frac{  \maxdv \maxdf^\theta \Cbdistort^2 }
            { \lambda (\lfoth)  }   
            ,
    \end{gather*}
    where $\Cdvtmp$ is a local $\eps$-\holder constant for $v'$ at point $f^N(x)$.
\end{lm}

\begin{proof}[Proof of the lemma]

    First estimate from the statement follows from the upper bound on $h$:
    \begin{gather*}
        h^{-\theta}
        \enbrace{\frac{\delta_1}{\Cbdistort}}^{\theta-1}
        \vmod{ S_{0,N}(x,h) }
        =
        h^{1-\theta} 
        \enbrace{\frac{\delta_1}{\Cbdistort}}^{\theta-1}
        \vmod{
            \sum_{i=0}^N
             v'(f^i(\xi_i))
        \frac{ \dfxx{i}{\xi_i}   }{\dfxthx{i}{x }} 
        \frac{ 1 }{ \dfxthxp{ f^i(x) } } 
    }                                   
              \leq 
              \\
              \leq
              \Cbdistort^\theta \lambda^{-\theta}
        \vmod{
            \sum_{i=0}^N
             v'(f^i(\xi_i))
        \frac{ \dfxoth{i}{\xi_i}   }{\dfxoth{N}{x }} 
    }
            \leq 
              \Cbdistort \lambda^{-\theta}
        \vmod{
            \sum_{i=0}^N
             v'(f^i(\xi_i))
        \frac{ 1   }{\dfxoth{N-i}{x }} 
    }
    \leq
    \\
    \leq
            \frac{2 \Cbdistort \maxvprime}{ \lambda^\theta (\lfoth)    } 
            .
    \end{gather*}

        Note that we have for certain $\xi\in (x,\xi_N) \subset (x,x+h)$
    \begin{gather*}
        \vmod{ v'(f^{  N}(\xi_N)) - 
        v'(f^{ N}(x)) }
        \leq  
        \Cdvtmp \vmod{ f^{  N}(\xi_N) - f^{  N}(x) }^\eps
        =
        \\
        =
        \Cdvtmp \vmod{ \dfxx{  N}{\xi} (\xi_N - x)}^\eps 
        =
        \Cdvtmp  \vmod{ \Cbdistort 
            \dfxx{ N}{x} (\xi_N - x)}^\eps 
        \leq 
        \Cdvtmp  \delta_2^\eps .
    \end{gather*}
    We also have 
        $\vmod{ v'(f^{ N}(x) )  - v'(c) } \leq 
        \Cdvtmp \delta_2^\eps$.
    Using the distortion estimate we get
    \begin{gather*}
        h^{-1} \vmod{ {S}_{0,N}(x,h) } 
        =
        \vmod{
        \sum_{i=0}^{N} v'(f^i(\xi_i))
        \frac{ \dfxx{i}{\xi_i}   }{\dfxthx{i}{x }} 
        \frac{ 1 }{ \dfxthxp{ f^i(x) } } 
    } \geq 
    \\
        \geq
            \vmod{v'(f^N(\xi_N))}
            \frac{ \dfxx{N}{\xi_N}   }{\dfxthx{N}{x }} 
            \frac{ 1 }{ \dfxthxp{ f^N(x) }  }  
            -
            \sum_{i=0}^{N-1} \vmod{v'(f^i(\xi_i))}
            \frac{ \dfxx{i}{\xi_i}   }{\dfxthx{i}{x }} 
            \frac{ 1 }{ \dfxthxp{ f^i(x) }  } 
            \geq
            \\
            \geq 
            \vmod{v'(f^N(\xi_N))}
            \frac{ \dfxoth{N}{\xi_N}   }
            { \Cbdistort^{\theta} \maxdf^\theta  }  
            -
            \sum_{i=0}^{N-1} \vmod{v'(f^i(\xi_i))}
            \frac{ \Cbdistort^\theta \dfxoth{i}{\xi_i}   }{ \lambda^\theta  } 
            .
    \end{gather*}

%

    Using \eqref{eq:dfdivide} and
    the lower bound on $h$ 
     and \eqref{eq:vDiff}
    we can write the following estimates:
    \begin{gather*}
        h^{-\theta}
        \vmod{ {S}_{0,N}(x,h) }
    \geq
    \vmod{v'(f^{N}(\xi_N)) }
            \frac{ \delta_1^{1-\theta} }
            { \Cbdistort^{2(1-\theta)+\theta} \maxdf^\theta }
            -
            \lambda^{-\theta} \sum_{i=0}^{N-1} \vmod{v'(f^i(\xi_i)) }
            \frac{\Cbdistort^{\theta}  \delta_1^{1-\theta} }
            {  \dfxoth{N-i}{f^{i}(\xi_i)  } }  
            \geq 
            \\ 
            \geq  
            \vmod{v'(f^{N}(x)) } 
            \frac{ \delta_1^{1-\theta} }
            { \Cbdistort^{2(1-\theta)+\theta} \maxdf^\theta }
            -
            \Cdvtmp 
            \frac{ \delta_1^{1-\theta} \delta_2^\eps }
            { \Cbdistort^{2(1-\theta)+\theta} \maxdf^\theta }
            -
            \frac{\Cbdistort^{\theta} \maxdv}
                {  \lambda^{\theta} } 
            \sum_{i=0}^{N-1} 
            \frac{  \delta_1^{1-\theta} }
            {  \dfxoth{N-i}{f^{i}(\xi_i)  } }  
            \geq 
            \\
            \geq
            v'(c)
            \frac{ \delta_1^{1-\theta} }
            { \Cbdistort^{2-\theta} \maxdf^\theta }
            -
            \Cdvtmp 
            \frac{ \delta_1^{1-\theta} \delta_2^\eps }
            { \Cbdistort^{2-\theta} \maxdf^\theta }
            -
            \Cdvtmp 
            \frac{ \delta_1^{1-\theta} \delta^\eps }
            { \Cbdistort^{2-\theta} \maxdf^\theta }
            -
            \Cbdistort^{\theta} \delta_1^{1-\theta} 
            \frac{  \maxdv }
            { \lambda (\lfoth)  }   
            .
            \end{gather*}
        \end{proof}

    \subsubsection{Upper bound}
    First we prove that $\alpha$ is $\theta$-\holder.
    For every $x$ for every $\lambda^{-1}/(2 \Cbdistort) >h>0$
    take  a natural number $N$ such that
    $ { 1/2  } \leq 
    h { \Cbdistort\dfxx{N+1}{x} }
    \leq
    { 1 }    $.
    Then lemmas \ref{lm:nlin_fderiv_bounds}, \ref{lm:nlin_supp_bounds}
    and \ref{lm:nlin_lower_bd} 
    for 
    $\delta_1 = 1/2$ and  $\delta_2 =1$ 
    imply that
    the following upper bound for the normalized absolute value
    of \eqref{eq:alphadif_nonlin_large_lbd} holds:
    \begin{gather*}
        {\vmod{\alpha(x+h) - \alpha(x) } }{h^{-\theta}} 
        \leq 
        C
         +
         C' N h^{1-\theta}  
         .
    \end{gather*}
    where $C,C'>0$ do not depend on $h$ and $N$.

    As this bound does not depend on $h$ and $N$ (the last term
    is negligible because $h$ is exponentially small in $N$), it
    proves that $\alpha$ is $\theta$-\holder.


    \subsubsection{Lower bound when
    condition (A) holds for $\powk = 1$ }
    \label{sssec:lb_pow1}

    Assume that condition (A) is satisfied for 
    a point $c$ and $\powk = 1$. 
    Suppose without restricting generality that $v'(c)>0$.

    Fix $\hat{h}$.
    We first give expressions for $\delta_1,\delta_2$
    that depend only on $f$ and $v$
    and later
    select $x$, $N$ and $h = h(N,\delta_1,\delta_2)$
    such that $ \vmod{\alpha(x) - \alpha(x+h) } / h^\theta$ 
    has a positive lower bound.

    For every
    $0<\delta_1,\delta_2\leq 1$
    for every $x$ for which  
    there exists a natural number $N>1$ such that
      $  \vmod{f^{ N}(x) - c} < 
        \delta_2 
        $,
    the lemmas \ref{lm:nlin_fderiv_bounds},
    \ref{lm:nlin_supp_bounds} and \ref{lm:nlin_lower_bd}
    imply that
    for every
        $h = h(N)$ such that 
        ${  \delta_1} 
        \leq 
        h  { \Cbdistort \dfxx{N}{x} }
        \leq
        { \delta_2 }   $,
    the following lower bound  for the normalized
    absolute value
    of \eqref{eq:alphadif_nonlin_large_lbd} holds:
    \begin{gather}
        \frac{\vmod{\alpha(x+h) - \alpha(x) } }{h^\theta} 
        \frac{\maxdf^{\theta} \Cbdistort^{2-\theta} } 
        {\delta_1^{1-\theta}}  
        \geq 
            v'(c)
            -
            2\Cdvtmp 
            \delta_2^\eps 
            -
            \frac{  \maxdv \maxdf^\theta \Cbdistort^{2} }
            { \lambda (\lfoth)  }   
                -
                \nonumber
                \\
                -      
                \enbrace{ \frac{\maxdf}{\lambda^{2}} }^\theta  
                \frac{2 \Cbdistort^{2-\theta}  \maxv }{  \lfth  } \delta_1^{-1}  
                \nonumber
            -   \\
            -
        \frac{\delta_1^{1-\theta}}  {\maxdf^{\theta} \Cbdistort}
            \enbrace{
        \frac{ \enbrace{\lambda^{-s\theta} + \Cbdistort^{\theta}
            \enbrace{ \frac{\maxdf}{\lambda^2}}^{s\theta} 
        }  
    \maxv }{ \lfth } 
        \delta_1^{-\theta} h^\theta 
         +
         \Csecest(s) N h^{1-\theta}  }
         .
         \label{neq:lbound_nlin_large}
    \end{gather}

Put
\begin{gather*}
    \delta_1 = \enbrace{ \frac{\maxdf}{\lambda^{2}} }^\theta 
    \frac{6 \maxv \Cbdistort^{2-\theta}}{ (\lfth) \dvc }
    ,\quad 
    \delta_2 = \enbrace{ \frac{\dvc}{6\Cdv}  }^{1/\eps} 
    .
\end{gather*}
Inequalities \eqref{neq:A_bd1},\eqref{neq:A_bd3} from condition (A)
imply that $\delta_1\leq \delta_2\leq 1$.

Now fix a point $x$ such that there exists
a natural number $N$ such that
\begin{gather*}
        \vmod{f^{ N}(x) - c} < 
        \delta_2;\quad 
        { \delta_2 } < \hat{h} {\dfxx{N}{x} }
    .
\end{gather*}
One can take $x$ to be a point from
$N$'th-preimage of $B_{\delta_2}(c)$ for sufficiently large $N$ or
any point with a
dense trajectory (the set of which has full Lebesgue measure).
In the latter case $x$ does not depend on $\hat{h}$.

Now inequality \eqref{neq:A_bd2} from condition (A)
and the definition of $\delta_1,\delta_2$
allow us to select an $h < \hat{h}$ such that 
    \begin{gather*}
        \frac{\vmod{\alpha(x+h) - \alpha(x) } }{h^\theta} 
        \frac{\maxdfper^{\theta} \Cbdistort^{2-\theta}} 
        {\delta_1^{1-\theta}}  
        \geq 
        \\
        \geq
        \frac{ \dvc }{12}
        -
        \frac{\delta_1^{1-\theta}}  {\maxdf^{\theta} \Cbdistort}
            \enbrace{
        \frac{ \enbrace{\lambda^{-s\theta} + \Cbdistort^{\theta}
            \enbrace{ \frac{\maxdf}{\lambda^2}}^{s\theta} 
        }  
    \maxv }{ \lfth } 
        \delta_1^{-\theta} h^\theta 
         +
         \Csecest(s) N h^{1-\theta}  } .
    \end{gather*}

    Choosing $s$ large enough and 
    increasing 
    $N$ if necessary (depending on the choice of $s$) we can 
make the last summand less than  $\dvc/24$ because of pinching condition.



\end{proof}

\begin{proof}[Proof of Theorem \ref{cor:powk1}]
    Put $c$ to be a point where the maximum of $v'$ is attained.
    This implies that $v''(c) = 0$.

    A modification of the proof ot 
    Proposition \ref{thm:thetwist_expcircle}
    gives the necessary result.
    

    First we modify a proof of 
    Lemma \ref{lm:nlin_lower_bd}.
    Note that we have for 
    certain $\xi\in (x,\xi_N) \subset (x,x+h)$
    \begin{equation*}
        \vmod{ f^{  N}(\xi_N) - f^{  N}(x) }
        =
         \vmod{ \dfxx{  N}{\xi} (\xi_N - x)} 
        \leq 
         \Cbdistort \dfxx{  N}{x} h 
         \leq 
         \delta_2
    \end{equation*}

    Let $\Gamma$ is the local $\eps$-\holder constant
    of $v''$ at point $c$ (morally it is a third derivative of $v$
    at $c$).
    Thus
    \begin{gather*}
        \vmod{ v'(f^{  N}(\xi_N)) - 
        v'(f^{ N}(x)) }
        \leq  
        \\
        \leq
        \enbrace{v''(f^N(x)) +  
        \Cdvtmp \vmod{ f^{  N}(\xi_N) - f^{  N}(x) }^\eps
          }
        \vmod{ f^{  N}(\xi_N) - f^{  N}(x) }
        \leq
        \\
        \leq
        (v''(c) + \Cdvtmp\vmod{\dfxx{N}{x}-c}^\eps + 
        \Cdvtmp\delta_2^\eps )
        \delta_2
        \leq 
        2\Cdvtmp \delta_2^{1+\eps}
        .
    \end{gather*}

    Note also that
    \begin{equation*}
        \vmod{ v'(f^{ N}(x) )  - v'(c) } 
        \leq 
        \Cdvtmp \delta_2^\eps
        \vmod{ f^{ N}(x)   - c } 
        \leq 
        \Cdvtmp \delta_2^{1+\eps}
        .
    \end{equation*}

    Then 
    the last estimate from 
    the proof of Lemma \ref{lm:nlin_lower_bd}
    implies the following:
    \begin{gather*}
        h^{-\theta}
        \frac{ \delta_1^{1-\theta} }
        { \maxdf^{\theta} \Cbdistort^{2-\theta} }
        \vmod{ S_{N} }
        \geq
            v'(c)
            -
            3\Cdvtmp \delta_2^{1+\eps}
            -
            \frac{  \maxdv \maxdf^\theta \Cbdistort^2 }
            { \lambda (\lfoth)  }   
            .
    \end{gather*}

    Then the esitimate
    \eqref{neq:lbound_nlin_large} changes as well:
    \begin{gather*}
        \frac{\vmod{\alpha(x+h) - \alpha(x) } }{h^\theta} 
        \frac{\maxdf^{\theta} \Cbdistort^{2-\theta} } 
        {\delta_1^{1-\theta}}  
        \geq 
            v'(c)
            -
            3\Cdvtmp 
            \delta_2^{1+\eps}
            -
            \frac{  \maxdv \maxdf^\theta \Cbdistort^{2} }
            { \lambda (\lfoth)  }   
                -
                \\
                -      
                \enbrace{ \frac{\maxdf}{\lambda^{2}} }^\theta  
                \frac{2 \Cbdistort^{2-\theta}  \maxv }{  \lfth  } \delta_1^{-1}  
            -
        \frac{\delta_1^{1-\theta}}  {\maxdf^{\theta} \Cbdistort}
            \enbrace{
        \frac{ \enbrace{\lambda^{-s\theta} + \Cbdistort^{\theta}
            \enbrace{ \frac{\maxdf}{\lambda^2}}^{s\theta} 
        }  
    \maxv }{ \lfth } 
        \delta_1^{-\theta} h^\theta 
         +
         \Csecest(s) N h^{1-\theta}  }
         .
    \end{gather*}

    Put
    \begin{gather*}
        \delta_1 = \enbrace{ \frac{\maxdf}
        {\lambda^{2}} }^\theta 
        \frac{6 \maxv \Cbdistort^{2-\theta}}{ (\lfth) \dvc }
        ;\quad
        \delta_2 = 
        \enbrace{ \frac{\dvc}{9\Cdvtmp}  }^{1/{1+\eps} } 
        .
    \end{gather*}

        It is easy to see that
        $\frac{\delta_1}{\delta_2}$ is proportional to
          $  {\maxv \Cdvtmp}/{ (v'(c))^{1+\frac{1}{1+\eps}} }$.
        Note then if one multiplies $v$ by a constant
        its \holder exponents do not change.
        Therefore we can assume (multiplying $v$
        by a small enough constant) that 
        $\delta_1 \leq \delta_2$.

        Condition \eqref{neq:simple_bd1}
        from the statement of 
        the Corollary implies that $\delta_2 \leq 1$.
        Finally condition \eqref{neq:simple_bd2} implies that 
        \begin{equation*}
            \frac{  \maxdv \maxdf^\theta \Cbdistort^{2} }
            { \lambda (\lfoth)  }   
            \leq \frac{v'(c)}{4}
            .
        \end{equation*}

        Now we can literally repeat the arguments
        from 
        Subsection \ref{sssec:lb_pow1}
        to conclude.

\end{proof}

\section{Acknowledgements}

This research was supported by the
Chebyshev Laboratory  (Department of Mathematics and
Mechanics, St. Petersburg State University)
[under RF Government grant 11.G34.31.0026];
 JSC "Gazprom Neft";
  St. Petersburg State University [thematic project 
  6.38.223.2014].

The author is grateful to Viviane Baladi and Daniel Smania
for advice and discussions.

\section{Appendix}

Here we prove 
the rest of the
Proposition \ref{thm:thetwist_expcircle}
         -- a lower bound for 
    $$\vmod{\alpha(x)-\alpha(x+h)}h^{-\theta}$$
for general $\powk\geq 1$.

\subsection{Proof of the lower bound for general expanding $f$, general $\powk$}

For a natural number $\powk \geq 1$
    for every 
    $x$ from $S^1$ and $h>0$
 we can rewrite the formula \eqref{eq:cohomSolFormula} as
\begin{gather*}
    \alpha(x) = - \sum_{j=0}^{\powk-1}
    \sum_{i=0}^{\infty} \frac{v(f^{\powk i+j}(x)) }{ \dfxth{\powk i+j+1} }
    .
\end{gather*}

Therefore
\begin{gather*}
    \alpha(x) - \alpha(x+h) = \\
    =
    \sum_{j=0}^{\powk-1}  B_j(x,h) 
        + \sum_{i=0}^{\infty} v(f^i(x+h))\enbrace{\frac{1}{\dfxthx{i+1}{x+h} }
        - \frac{1}{\dfxthx{i+1}{x} } } .  \nonumber 
\end{gather*}
where
\begin{gather*}
    B_j(x,h) =  
        {
        \sum_{i=0}^{\infty} \frac{v(f^{\powk i+j}(x)) }{ \dfxthx{\powk i+j+1}{x} } 
    -
    \sum_{i=0}^{\infty} \frac{v(f^{\powk i+j}(x+h)) }{ \dfxthx{\powk i+j+1}{x}} 
        }
        =
        \\
        = \sum_{i=0}^{\infty}
        \frac{1}{\dfxthx{ (j+\powk i+1)}{x} }
        \enbrace{v(f^{j+\powk i}(x+h)) -v(f^{j+\powk i}(x))}   .
    \end{gather*}

    For $n\in \Zplus$ and $m \in \Zplusinf$ such that $n\leq m$
    we introduce the following notation:
    \begin{equation*}
        S^{(j)}_{n,m}(x,h) = 
        \sum_{i=n}^{m} \frac{1}{\dfxthx{j+\powk i+1}{x} }
        \enbrace{v(f^{\powk i}(x+h)) -v(f^{\powk i}(x))}     .
%
    \end{equation*}

    Note that for every $0\leq j \leq \powk-1$ for every $n\in\Zplus$ there
    exist $\xi_i = \xi_i^{(j)} \in (x,x+h)$ for $0\leq i\leq n$
    such that
    \begin{gather*}
        S_{0,N}(x,h)   
        =
        {h} 
        \sum_{i=0}^{N} v'(f^{j+\powk i}(\xi_i))
        \frac{ \dfxx{j+\powk i}{\xi_i}   }{\dfxthx{j+\powk i+1}{x }} 
        =            \\
        =
        {h} 
        \sum_{i=0}^{N} v'(f^i(\xi_i))
        \frac{ \dfxx{j+\powk i}{\xi_i}   }{\dfxthx{j +\powk i}{x }} 
        \frac{ 1 }{ \dfxthxp{ f^{j+\powk i}(x) } }
        .
    \end{gather*}

    \subsubsection{Technical lemmas}

\begin{lm}        \label{lm:nlin_small_supp_bounds}
    Let $x\in\mfd$, let $N$ be a natural number,  $0\leq j\leq \powk-1$,
    $0<\delta_1,\delta_2\leq 1$
    and $h>0$ be such that
    \begin{gather*} 
          \delta_1 \leq h { \Cbdistort \dfxx{\powk N}{x} }
        .
    \end{gather*}

    Then 
    the following estimate holds:
\begin{gather*}
    \vmod{S^{(j)}_{N+1,\infty}(x,h) } \leq
    \frac{2 
\maxv}{ \lambda^{(\powk+j+1)\theta}
    ( \lfthk ) } \delta_1^{-\theta} h^\theta  
    .
\end{gather*}
\end{lm}

\begin{proof}[Proof of the lemma]

We may easily estimate the tail
of ${S_{0,\infty} (x,h)}$ using the lower bound on $h$:
\begin{gather*}
    \vmod{S^{(j)}_{N+1,\infty} (x,h) }  \leq
    2\maxv \frac{1}{\dfxthx{\powk (N+1)+j+1}{x } } \sum_{i=0}^{\infty} 
    \lambda^{-i\powk\theta} 
    = \\
    =  \frac{2\maxv}{1 - \lambda^{-\powk\theta}} 
    \frac{1}{ \dfxthx{\powk + j+1}{f^{\powk N}(x )} 
        \dfxthx{\powk N}{x } }
    \leq 
    \\
    \leq
    \frac{2 
\maxv}{ \lambda^{(\powk+j+1)\theta}
    ( \lfthk ) } \delta_1^{-\theta} h^\theta  
    .
\end{gather*}

\end{proof}

    \begin{lm}        
    \label{lm:nlin_lower_bd_small_lbd}

    Let $x\in\mfd$, let $N$ be a natural number, $0\leq j\leq \powk-1$,
    $0<\delta_1,\delta_2 \leq 1$, $\delta>0$
    and $h>0$ be such that
    \begin{gather*} 
        \frac{  \delta_1}{ \Cbdistort \dfxx{\powk N}{x} } 
        \leq h 
        \leq
        \frac{  \delta_2 }{ \Cbdistort \dfxx{\powk N}{x} }   
        .
    \end{gather*}
    Suppose also that $0\leq j\leq \powk -1 $ 
    and
    \begin{equation}
        \delta_2 < \maxdf^{-j}
        \label{neq:cond_to_use_bd}
        .
    \end{equation}

    If $x$ is a point of $\mfd$
    such that
    $ \dist(f^{j + \powk N}(x), f^j(c) ) \leq \delta$
    and
    $ v'(f^{j}(c)) > 0$
    then
    \begin{gather*}
        h^{-\theta} \vmod{ S^{(j)}_{0,N}(x,h) } 
        \frac
        { \Cbdistort^{2-\theta} \maxdf^\theta  }  
        {\delta_1^{1-\theta}  \lambda^{j(1-\theta)}}
        \geq
         v'(f^{j} (c)) -
        \Cdvtmp \delta_2^\eps  \maxdf^{j\eps} 
        - \Cdvtmp \delta^\eps
        -
        \\
        -
        \frac{ \maxdv \Cbdistort^{2}
         }
        { \lambda^{\powk(1-\theta)+\theta} (\lfothk)  }  
        \frac{ \maxdf^{j(1-\theta) +\theta} }
        { \lambda^{j(1-\theta)} },
    \end{gather*}
    where $\Cdvtmp$ is a local $\eps$-\holder constant for $v'$
    at point $f^{j+\powk N}(x)$.
\end{lm}

\begin{proof}[Proof of the lemma]

        Note that
        using the upper bound on $h$
        we obtain for some $\xi\in (x,\xi_N) \subset (x,x+h)$
    \begin{gather*}
        \vmod{ v'(f^{j+ \powk N}(\xi_N)) - 
        v'(f^{j+\powk N}(x)) }
        \leq  
        \Cdvtmp \vmod{ f^{j+ \powk N}(\xi_N) - f^{ j +\powk N}(x) }^\eps
        =
        \\
        =
        \Cdvtmp \vmod{ \dfxx{j + \powk N}{\xi} (\xi_N - x)}^\eps 
        =
        \Cdvtmp  \vmod{ \dfxx{j}{f^{\powk N} (\xi) }
            \dfxx{\powk N}{\xi} (\xi_N - x)}^\eps 
            \leq
        \Cdvtmp  \maxdf^{j\eps} \delta_2^\eps .
    \end{gather*}
    We also have that
    $    \vmod{ v'(f^{j+\powk N}(x) )  - v'(f^j(c)) } 
        \leq 
        \Cdvtmp \delta^\eps$.

    Inequality \eqref{neq:cond_to_use_bd} allows
    us to use the distortion estimates for all 
    indices $0\leq i \leq j+\powk N$, therefore
    we get 
    \begin{gather*}
        h^{-1} \vmod{ S^{(j)}_{0,N}(x,h) } 
        =
        \vmod{
            \sum_{i=0}^{N} v'(f^{j+\powk i}(\xi_i))
        \frac{ \dfxx{j+\powk i}{\xi_i}   }
            {\dfxthx{j+\powk i}{x }} 
            \frac{ 1 }{ \dfxthxp{ f^{j +\powk i}(x) } } 
    } \geq 
    \\
        \geq
        \vmod{v'(f^{j+\powk N}(\xi_N))}
            \frac{ \dfxx{j+\powk N}{\xi_N}   }
            {\dfxthx{j+\powk N}{x }} 
            \frac{ 1 }{ \dfxthxp{ f^N(x) }  }  
            -
            \\
            -
            \sum_{i=0}^{N-1} \vmod{v'(f^{j+\powk i}(\xi_i))}
            \frac{ \dfxx{j+\powk i}{\xi_i}   }
                 {\dfxthx{j+\powk i}{x }} 
                 \frac{ 1 }{ \dfxthxp{ f^{j+\powk i}(x) }  } 
            \geq
            \\
            \geq 
            \vmod{v'(f^{j+\powk N} (\xi_N))}
            \frac{ \dfxoth{j+\powk N}{\xi_N}   }
            { \Cbdistort^{\theta} \maxdf^\theta  }  
            -
            \sum_{i=0}^{N-1} \vmod{v'(f^{j+\powk i}(\xi_i))}
            \frac{ \Cbdistort^\theta 
                \dfxoth{j+\powk i}{\xi_i}   }
            { \lambda^\theta  } 
    \end{gather*}

    Using \eqref{eq:dfdivide}, 
    the lower bound on $h$, 
     and \eqref{eq:vDiff},
    we can write the following estimates:

    \begin{gather*}
        h^{-\theta} \vmod{ S^{(j)}_{0,N}(x,h) } 
            \geq 
            \delta_1^{1-\theta} 
            \frac{v'(f^{j+\powk N} (\xi_N))}
        { \Cbdistort \maxdf^\theta  }  
        \frac{ \dfxoth{j+\powk N}{\xi_N}   }
        { \dfxoth{\powk N}{x} }
        -
        \\
        -
            \delta_1^{1-\theta} 
        \sum_{i=0}^{N-1} 
        \frac{\vmod{v'(f^{j+\powk i}(\xi_i))} }
        { \lambda^\theta \Cbdistort^{1-2\theta} } 
        \frac{  
            \dfxoth{j+\powk i}{\xi_i}   }
        { \dfxoth{\powk N}{x} }
        \geq
        \\
        \geq
            \delta_1^{1-\theta} 
            \frac{ \vmod{v'(f^{j+\powk N} (\xi_N))} }
            { \Cbdistort^{2-\theta} \maxdf^\theta  }  
            \dfxoth{j}{f^{\powk N}(\xi_N)}   
            -
            \\
            -
            \delta_1^{1-\theta} 
        \sum_{i=0}^{N-1} 
        \frac{ \maxdv \Cbdistort^{\theta}}
        { \lambda^\theta  } 
        \frac{ \dfxoth{j}{f^{\powk i}(\xi_i) }   }
        {\dfxoth{\powk(N-i)}{f^{\powk i}(\xi_i) }}
        \geq
        \\
        \geq
        \delta_1^{1-\theta} 
        \frac{ \vmod{v'(f^{j+\powk N} (\xi_N))} }
        { \Cbdistort^{2-\theta} \maxdf^\theta  }  
        \lambda^{j(1-\theta)}
        -
        \delta_1^{1-\theta} 
        \frac{ \maxdv \Cbdistort^{\theta}}
        { \lambda^\theta  } 
        \sum_{i=0}^{N-1} 
        \frac{\maxdf^{j(1-\theta)} }
        {\lambda^{ \powk (N-i) (1-\theta) } }
        \geq
        \\
        \geq
        \frac{\delta_1^{1-\theta}  }
        { \Cbdistort^{2-\theta} \maxdf^\theta  }  
        \lambda^{j(1-\theta)}
        \enbrace{ v'(f^{j} (c)) -
        \Cdvtmp \delta_2^\eps \maxdf^{j\eps} 
        - \Cdvtmp \delta^\eps
        } 
        -
        \frac{ \maxdv \Cbdistort^{\theta}
        \delta_1^{1-\theta} }
        { \lambda^{\powk(1-\theta)+\theta} (\lfothk)  }  
        \maxdf^{j(1-\theta) }
        .
    \end{gather*}
        
\end{proof}



    Assume that condition (A) is satisfied for 
    a point $c$ and $\powk\geq 1$.
    Suppose without restricting generality that $v'(c)>0$ (this
    implies that $\dvcj > 0$ for every $0\leq j\leq \powk-1$ by condition (A) ).

    Fix $\hat{h}$.
    We first give expressions for $\delta_1,\delta_2$
    that depend only on $f,v$ and $\powk$
    and later
    select $x$, $N$ and $h = h(N,\delta_1,\delta_2)$
    such that $ \vmod{\alpha(x) - \alpha(x+h) } / h^\theta$ 
    has a lower bound.

    For every
    $0<\delta^{(j)}_1,\delta^{(j)}_2 \leq \maxdf^{-j}$
    for every $x$ for which  
    there exists a natural number $N>\powk$ such that
      $  \vmod{f^{\powk N}(x) - c} < 
        \delta^{(j)}_2 \maxdf^{-\powk+1}$,
    lemmas above (\ref{lm:nlin_small_supp_bounds}
    and \ref{lm:nlin_lower_bd_small_lbd} for $\delta=\delta_2\maxdf^{j}$) 
    imply that
        for every
        $h = h(N)$ such that 
    \begin{gather*} 
          \delta^{(j)}_1
        \leq h  { \Cbdistort \dfxx{N}{x} }
        \leq
          \delta^{(j)}_2 ,
    \end{gather*}
    the following bound holds for every $0\leq j\leq \powk-1$:
        \begin{gather}
         h^{-\theta}
            B_j(x,h) 
        \frac
        { \Cbdistort^{2-\theta} \maxdf^\theta  }  
        { (\delta^{(j)}_1)^{1-\theta}  \lambda^{j(1-\theta)}}
        \geq      \nonumber
         v'(f^{j} (c)) -
        2\Cdv (\delta^{(j)}_2)^\eps 
         \maxdf^{j\eps} 
        -
                    \\
        -
        \frac{ \maxdv \Cbdistort^{2}
         }
        { \lambda^{\powk(1-\theta)+\theta} (\lfothk)  }  
        \frac{ \maxdf^{j(1-\theta) +\theta} }
        { \lambda^{j(1-\theta)} }
                -  \nonumber
                \\
                -
                \frac{2 \Cbdistort^{2-\theta}  \maxdf^\theta
        \maxv}{ \lambda^{(\powk+1)\theta+j}
    ( \lfthk ) } (\delta^{(j)}_1)^{-1} 
        \label{neq:final_lower_bd_nonlin}
        .
        \end{gather}
        Note that the right-hand side of 
        this expression does not depend on $x,N$ or $h$.

        It is possible to choose $\delta^{(j)}_1,\delta^{(j)}_2>0$ in such a way
    that the expression above is greater than zero.
    To guarantee that the second and the fourth summands in the right-hand
    side of \eqref{neq:final_lower_bd_nonlin} are both less than
    $\dvcj/3$ it is enough to put 
    \begin{gather*}
        \delta^{(j)}_1 \leq D^{(j)}_1 = \frac{6 \maxv \Cbdistort^{2-\theta} \maxdf^\theta }
        {(\lfthk)   \lambda^{(\powk+1)\theta + j} \dvcj },
        \\
        \delta^{(j)}_2 \geq D^{(j)}_2 = \enbrace{ \frac{\dvcj}{6\Cdv}}^{1/{\eps}}
        \frac{1}{\maxdf^{j }  }
        .
    \end{gather*}

    To be able to select $\delta_1\leq \delta_2$ in such a way
    so that
    they satisfy bounds above
    for every $j$
    but do not depend on $j$,
    we use inequality \eqref{neq:A_bd1} from condition (A).
    It exactly means that 
    \begin{equation*}
        \frac{6 \maxv \Cbdistort^{2-\theta} \maxdf^\theta }
        {(\lfthk)   \lambda^{(\powk+1)\theta } \mindvcj } 
        =  \max_j D^{(j)}_1 \leq
        \min_j D^{(j)}_2
        =
         \enbrace{ \frac{\mindvcj}{6\Cdv}}^{1/{\eps}}
        \frac{1}{\maxdf^{\powk - 1 }  }
        ,
    \end{equation*}
    where 
        $\mindvcj = \min_{0 \leq j \leq \powk-1} {v'(f^j(c))}$

    Put 
    \begin{gather*}
        \delta_1 = \max_{0 \leq j \leq \powk-1} D^{(j)}_1,
        \nonumber \quad 
        \delta_2 = \min_{0 \leq j \leq \powk-1} D^{(j)}_2. 
    \end{gather*}
    Inequality \eqref{neq:A_bd3} from condition (A) 
    implies that $\delta_2 \leq \maxdf^{-\powk+1}$.

    Now fix a point $x$ such that there exists
    a natural number $N$ such that
    \begin{gather*}
        \vmod{f^{\powk N}(x) - c} < 
        \delta_2 \maxdf^{-\powk+1}; 
        \quad
        \frac{  \delta_2 }{ \Cbdistort \dfxx{N}{x} }   
         < \hat{h}
         .
    \end{gather*}
    One can take $x$ to be a point from the
    $\powk N$'th-preimage of $B_{\delta_2}(c)$ for sufficiently large $N$ or
    any point with a
    dense trajectory (the set of which has full Lebesgue measure).
    In the latter case $x$ does not depend on $\hat{h}$.

    Now inequality \eqref{neq:A_bd2} from condition (A)
    and the definition of $\delta_1,\delta_2$
    allow us to select a (single) $h < \hat{h}$ such that 
    expression \eqref{neq:final_lower_bd_nonlin} is larger than
    $\dvcj / 12$ for every $j$.

    It implies that
    \begin{gather*}
        \frac{ \vmod{ \alpha(c) - \alpha(c+h) } } { h^\theta}
        \geq
        \sum_{j=0}^{\powk -1 } 
         \Cbdistort^{\theta-2} \delta_1^{1-\theta} 
                \frac{ \lambda^{j(1-\theta) } } 
                {\maxdf^{\theta} }  
                \frac{\dvcj }{12}
        - 
        \\
        -
        \frac{\delta_1^{1-\theta}}  {\maxdf^{\theta} \Cbdistort}
            \enbrace{
        \frac{ \enbrace{\lambda^{-s\theta} + \Cbdistort^{\theta}
            \enbrace{ \frac{\maxdf}{\lambda^2}}^{s\theta} 
        }  
    \maxv }{ \lfth } 
        \delta_1^{-\theta} h^\theta 
         +
     \Csecest(s) \powk N h  }
     .
    \end{gather*}

    Then as in the case of $\powk =1 $,
    by choosing $s$ large enough and 
    increasing 
    $N$ if necessary (depending on the choice of $s$), we can 
    make the last 
    summand (coming from the nonlinearity of $f$)
    less than a half of the first sum
    because of pinching condition.

    Finally 
    \begin{equation*}
        \frac{ \vmod{\alpha(c) - \alpha(c+h) } } { h^\theta}
        \geq
        C_0 > 0,
    \end{equation*}
    where
    \begin{gather*}
        C_0  =  \frac{1}{2} \sum_{j=0}^{\powk -1 } 
         \Cbdistort^{\theta-2} \delta_1^{1-\theta} 
                \frac{ \lambda^{j(1-\theta) } } 
                {\maxdf^{\theta} }  
                \frac{\dvcj }{12}    
                =    \nonumber
                \\
                =
                \frac { \Cbdistort^{\theta-2} }
                {24 \maxdf^{\theta}}
                \enbrace{ \frac{6 \maxv \Cbdistort^{2 -\theta}
         \maxdf^\theta }
         {(\lfthk)   \lambda^{(\powk+1)\theta } \mindvcj }  }^{1-\theta}
                 \sum_{j=0}^{\powk -1 } 
                { \lambda^{j(1-\theta) } }  
                {\dvcj }
                =            \nonumber
                \\
                =
                \frac{6^{1-\theta}  }
                { 24  \Cbdistort^{\theta(1-\theta)} 
                \maxdf^{\theta^2} }
                \enbrace{ \frac{ \maxv   }
         {  (\lfthk)  \mindvcj }  }^{1-\theta}
         \cdot
         \\
         \cdot
          \lambda^{(\powk+1)\theta(\theta-1) }
                 \sum_{j=0}^{\powk -1 } 
                { \lambda^{j(1-\theta) } }  
                {\dvcj }
                .
    \end{gather*}
    in particular, $C_0$ does not depend on $h,\hat{h}$.


\end{document}